\documentclass{amsart}
\usepackage{amsfonts}

\newtheorem{theorem}{Theorem}
\theoremstyle{plain}

\newtheorem{condition}{Condition}

\newtheorem{corollary}{Corollary}

\newtheorem{definition}{Definition}
\newtheorem{example}{Example}

\newtheorem{lemma}{Lemma}

\newtheorem{proposition}{Proposition}
\newtheorem{remark}{Remark}

\numberwithin{equation}{section}
\input{tcilatex}

\begin{document}
\title[orthogonal polynomial martingales]{On Markov processes with
polynomial conditional moments}
\author{Pawe\l\ J. Szab\l owski}
\address{Department of Mathematics and Information Sciences,\\
Warsaw University of Technology\\
ul Koszykowa 75, 00-662 Warsaw, Poland }
\email{pawel.szablowski@gmail.com}
\date{July 2012}
\subjclass{60J25, 60G44; Secondary 60G99, 33C47}
\keywords{polynomial regression, martingales, Markov processes, harnesses,
qudratic harnesses, orthogonal polynomials, Cholesky decomposition,
expansion of Radon-Nikodym derivative of one measure with respect to the
other,.Lancaster expansion.}

\begin{abstract}
We study properties of a subclass of Markov processes that have all moments
that are continuous functions of the time parameter and more importantly are
characterized by the property that say their $n-$th conditional moment given
the past is also a polynomial of degree not exceeding $n.$ Of course all
processes with independent increments with all moments belong to this class.
We give characterization of them within the studied class. We indicate other
examples of such process. Besides we indicate families of polynomials that
have the property of constituting martingales. We also study conditions
under which processes from the analysed class have orthogonal polynomial
martingales and further are harnesses or quadratic harnesses. We provide
examples illustrating developed theory and also provide some interesting
open questions. To make paper interesting for a wider range of readers we
provide short introduction formulated in the language of measures on the
plane.
\end{abstract}

\thanks{The author is grateful to the unknown referee for his valuable
remarks that helped to improve the paper.}
\maketitle

\section{Introduction}

The results we are presenting in this paper can be interpreted also from the
analytical point of view. They concern Markov processes and probability
measures. But from the analytical point of view Markov process it is nothing
else but the two sets of measures. One say $\mu (.,t)$ indexed by some index
set $I\ni t,$ usually subset of real line and the other $\eta (.,t;y,s)$ by
the points of the product $(t,y,s)\in I\times \limfunc{supp}(\mu
(.,s))\times I.$ Both measures are assumed to be probabilistic i.e. they are
nonnegative and normalized by $1$ and satisfy certain regularity conditions
of which the most important is the so called Chapman--Kolmogorov condition
that states that%
\begin{equation*}
\int_{\limfunc{supp}(\mu (.,t))}\eta (.,u;y,t)\eta (dy,t;z,s)=\eta (.,u;z,s),
\end{equation*}%
for all $s<t<u.$ There is also other condition that relates these two sets
of measures. Namely%
\begin{equation*}
\int_{\limfunc{supp}(\mu (.,t))}\eta (.,t;y,s)\mu (dy,s)=\mu (.,t).
\end{equation*}%
All conditions assumed as well as all the results of this paper can be
expressed in terms of these measures. For example conditions imposed on
these measures that define a subclass of interesting for us measures can be
easily expressed with the help of $\mu $ and $\eta $ in the following way:%
\begin{equation}
\forall n>0,s<t:\int x^{2n}\mu \left( dx,t\right) <\infty ,~~\int x^{n}\eta
(dx,t;y,s)=Q_{n}(y,s,t),  \label{zal}
\end{equation}
where $Q_{n}$ denotes certain polynomial in $y$ of degree not exceeding $n.$

In fact we will assume that all measures $\mu (.,t)$ will be identifiable by
its moments which is slightly stronger assumption than the first assertion
of (\ref{zal}). For example it is known that if $\exists \alpha >0\forall
t\in I:E\exp (\alpha \left\vert x\right\vert )d\mu \left( x,t\right) <\infty 
$ then measure $\mu $ is identifiable by moments. In fact there exist other
conditions assuring this. For details see e.g. \cite{Sim98}.

Finally we assume that for every $n,m\in \mathbb{N}$ function%
\begin{equation*}
\int \int |x|^{n}\left\vert y\right\vert ^{m}\eta (dx,t;y,s)\mu (dy,s)
\end{equation*}%
is a continuous function of $s$ and $t$ at least on the diagonal $%
s\allowbreak =\allowbreak t$.

The problems that we are going to solve in this paper are the following:

1. Is it possible to find a linear combination of monomials $x^{i};$ $%
i\allowbreak =\allowbreak 0,\ldots ,n$ i.e. to find a polynomial $p_{n}$
such that $\int p_{n}(x;t)\eta (dx,t;y,s)\allowbreak =\allowbreak p_{n}(y;s)$
for all $s<t.$ \emph{Existence of polynomial martingales} in the
probabilistic language.

2. Under what conditions $\int (x-y)^{n}\eta (dx,t;y,s)$ does not depend on $%
y$ for any natural $n.$ \emph{Independence of increments} in the
probabilistic language.

3. When polynomials defined in point 1. are orthogonal i.e. $\int
p_{n}(x;t)p_{m}(x;t)\mu (dx,t)\allowbreak =\allowbreak 0$ for $n\neq m.$ 
\emph{Existence of orthogonal polynomials martingales.}

4. When 
\begin{eqnarray*}
&&\int \int \int x\eta (dx,t;y,s)g(y,z)\eta (dz,u;x,t)\mu (dy,s)\allowbreak
\\
&=&\allowbreak \int \int L(y,z,s,u)g(y,z)\eta (dz,u;y,s)\mu (dy,s),
\end{eqnarray*}
where $L$ is a linear function of $y$ and $z,$ and $g(y,z)$ is any bounded
measurable function of $y$ and $z.$ \emph{Harness property }in the
probabilistic terminology.

5. \emph{\ }When 
\begin{eqnarray*}
&&\int \int \int x^{2}\eta (dx,t;y,s)g(y,z)\eta (dz,u;x,t)\mu
(dy,s)\allowbreak \\
&=&\allowbreak \int \int Q_{2}(y,z,s,u)g(y,z)\eta (dz,u;y,s)\mu (dy,s),
\end{eqnarray*}
where $Q_{2}$ is a quadratic function of $y$ and $z$ for all $s<t<u.$ \emph{%
Quadratic} \emph{harness property }in the probabilistic terminology.

We prefer however traditional probabilistic notation as more intuitive.

Hence we study a subclass of one dimensional Markov processes $\mathbf{%
X\allowbreak =\allowbreak }(X_{t})_{t\in I}$ defined on a finite or infinite
segment that has the property that all its conditional moments of say degree 
$n$ are polynomials of degree not exceeding $n.$ Poisson, Wiener,
Ornstein--Uhlenbeck processes or more generally $q-$Wiener and $(\alpha ,q)-$%
OU processes (described for example in more detail in \cite{Szab-OU-W} and
briefly in Subsection \ref{example}) are the examples of such processes.
Similar approach using polynomials to derive some properties of stochastic
processes was applied by Schoutens and Teugels in \cite{Scho98} to study L%
\'{e}vy processes or Cuchiero et. al. in \cite{Cuh12} to improve simulation.
Our approach is general, applicable to all Markov processes that have
marginal distributions identifiable by moments.

To be more specific let us assume the following:

Let $\mathbf{X\allowbreak =\allowbreak }(X_{t})_{t\in I}$ be a real
stochastic process defined on some probability space $(\Omega ,\mathcal{F}%
,P) $ where $I\allowbreak =\allowbreak \lbrack l,r]$ is some finite or
infinite segment of a real line. Cases $l\allowbreak =\allowbreak -\infty $
or $r\allowbreak =\allowbreak \infty $ are allowed. Let us also assume that
for $\forall t\in I$ : $\limfunc{supp}X_{t}$ contains infinite number of
points and that $\forall n\in \mathbb{N},$ $t\in \mathbb{R}$ : $E\left\vert
X_{t}\right\vert ^{n}<\infty .$

Let us denote also \newline
$\mathcal{F}_{\leq s}\allowbreak =\allowbreak \sigma (X_{t}:t\in \lbrack
l,s])$ , $\mathcal{F}_{\geq s}\allowbreak =\allowbreak \sigma (X_{t}:t\in
\lbrack s,r])$ and $\mathcal{F}_{s,t}\allowbreak =\allowbreak \sigma \left(
X_{v}:v\notin (s,t),v\in I\right) $.

Moreover let us assume that $\exists N:\forall 0<$ $n\leq N;$ $s\neq t\in I$
matrix $[\limfunc{cov}(X_{t}^{i},X_{s}^{j})]_{i,j=1,\ldots ,n}$ is
non-singular. Processes satisfying these assumptions will be called totally
linearly independent of degree $N$ (briefly $N-$TLI).

We will also assume that $\exists N,\forall m,j\leq N:EX_{t}^{m}X_{s}^{j}$
are continuous functions of $t,s\in I$ at least for $s$ $\allowbreak
=\allowbreak t$. Such processes will be called mean-square continuous of
degree $N$ (briefly $N-$MSC).

Let us remark that sequence of independent random variables indexed by some
discrete linearly ordered set are not TLI.

By $L_{2}(t)$ let us denote space spanned by real functions square
integrable with respect to one-dimensional distribution of $X_{t}.$ By our
assumptions in $L_{2}(t)$ there exists set of orthogonal polynomials that
constitute base of this space.

Thus the class of Markov processes that we will consider is a class of
stochastic processes that are $N-$TLI and $N-$MSC and moreover satisfying
the following condition:

\begin{equation}
\exists N\forall N\geq n\geq 1,s\leq t:E(X_{t}^{n}|\mathcal{F}_{\leq
s})\allowbreak =\allowbreak Q_{n}(X_{s},s,t),  \label{p_reg}
\end{equation}%
where $Q_{n}(x,s,t)$ is a polynomial of degree not exceeding $n$ in $x.$ We
will call this class of processes Markov processes with polynomial
regression of degree $N$ (briefly $N-$MPR process). If $N$ can be taken $%
\infty $ then we will talk of MPR processes. More precisely we should call
this class $N-$rMPR class i.e. right Markov processes with polynomial
regression. However until we will consider left (with the obvious meaning)
class of Markov processes we will use the name MPR class.

Since conditional expectation of every polynomial $Q_{n}(X_{t};t)$ (with
respect to $\mathcal{F}_{\leq s})$ of degree $n$ is a polynomial $\hat{Q}%
_{n}(X_{s};s,t)$ of degree not exceeding $n$ there is a natural question if
one can select a polynomial $p_{n}(x;t)$ in such a way that $%
E(p_{n}(X_{t};t)|\mathcal{F}_{\leq s})\allowbreak =\allowbreak
p_{n}(X_{s};s) $ i.e. that $(p_{n}(X;t),\mathcal{F}_{\leq t})_{t\in I}$ is a
martingale. One can also pose another natural question when $%
E(X_{t}-X_{s})^{j}|\mathcal{F}_{\leq s})$ for all $j=1,2,\ldots $ is non
random which would lead to the property of having independent increments.

We answer these questions in the next Section \ref{ogol}.

The main core of the paper is contained in two Sections \ref{ort_mar} and %
\ref{harness} where we study a subclasses of the MPR class for which there
exist the sequence of polynomial martingales that are also orthogonal with
respect to one dimensional marginal probability measure. More precisely we
will assume that there exists a sequence of polynomials $\left\{
p_{n}(x;t)\right\} _{n\geq 0}$ orthogonal with respect to the distribution
of $X_{t}$ and such that $(p_{n}(X_{t};t),\mathcal{F}_{\leq t})$ is a
martingale (for every $n).$ Such processes will be called processes with
orthogonal polynomial martingales (OPM-class). Under some mild regularity
conditions we are able to expand Radon--Nikodym derivative of conditional
distribution (i.e. transitional distribution) with respect to a marginal
distribution in a Fourier series in polynomials $\left\{ p_{n}(x;t)\right\}
. $ Following this expansion we see that processes of the OPM class
satisfying those regularity conditions are completely characterized by their
marginal distributions. Hence these processes form a nice regular class that
in our opinion is worth to study in detail. As a result we are able to
characterize harnesses and further quadratic harnesses within OPM-class. Our
nicest result (Theorem \ref{lin_har}) concerns necessary and sufficient
conditions for OPM process to be a harness. Theorem \ref{main} specifies
necessary and sufficient conditions for OPM process to a be quadratic
harness.

Open problem are collected in Section \ref{open}. Longer proofs are
collected in Section \ref{dowody}.

\section{Markov Processes with polynomial regression\label{ogol}}

Let us assume that $Q_{n}(x,s,t)\allowbreak =\allowbreak
\sum_{k=0}^{n}\gamma _{n,k}(s,t)x^{k}.$ Further let us define sequence of
lower-triangular matrices $\mathcal{A}_{n}(s,t)=\left[ \gamma _{i,j}(s,t)%
\right] _{i,j=0,\ldots n}$ and let us denote: $X_{t}^{(n)}\allowbreak
=\allowbreak \lbrack 1,X_{t},X_{t}^{2}.\ldots ,X_{t}^{n}]^{T}\allowbreak \in
\allowbreak \mathbb{R}^{n+1},$ $m_{n}(t)\allowbreak =\allowbreak EX_{t}^{n}$
, $\mathbf{m}_{n}(t)\allowbreak =\allowbreak \lbrack 1,m_{1}(t),\ldots
,m_{n}(t)]^{T}\allowbreak =\allowbreak EX_{t}^{(n)}$. Using this notation (%
\ref{p_reg}) can be written as:%
\begin{equation}
E(X_{t}^{(n)}|\mathcal{F}_{\leq s})\allowbreak =\allowbreak \mathcal{A}%
_{n}(s,t)X_{s}^{(n)}.  \label{mp_reg}
\end{equation}

Taking expectation of both sides of (\ref{mp_reg}) results in equality: $%
\mathbf{m}_{n}(t)\allowbreak =\allowbreak \mathcal{A}_{n}(s,t)\mathbf{m}%
_{n}(s).$ Let us also define two akin matrices namely $\mathbf{M}%
_{n}(t)\allowbreak =\allowbreak \left[ m_{i+j}(t)\right] _{i,j=0,\ldots ,n}$
and $\mathbf{C}_{n}(s,t)\allowbreak =\allowbreak \lbrack
EX_{t}^{i}X_{s}^{j}]_{i,j=0,\ldots ,n}.$ Notice that these matrices have the
following probabilistic interpretation: $\mathbf{M}%
_{n}(t)=EX_{t}^{(n)}(X_{t}^{(n)})^{T},~~\mathbf{C}_{n}(s,t)\allowbreak
=EX_{t}^{(n)}(X_{s}^{(n)})^{T}.$ Multiplying both sides of (\ref{mp_reg}) by 
$\left( X_{s}^{(n)}\right) ^{T}$ and taking expectation results in the
following equation:%
\begin{equation}
\mathbf{C}_{n}(s,t)\allowbreak =\allowbreak \mathcal{A}_{n}(s,t)\mathbf{M}%
_{n}(s).  \label{podst}
\end{equation}%
Let us also introduce the following variance covariance matrices: 
\begin{equation*}
\Sigma _{n}(s,t)\allowbreak =\allowbreak
E(X_{t}^{(n)}-m_{n}(t))(X_{s}^{(n)}-m_{n}(s))^{T}=\mathbf{C}_{n}(s,t)-%
\mathbf{m}_{n}(t)\mathbf{m}_{n}^{T}(s).
\end{equation*}%
Let subtract from both sides (\ref{mp_reg}) equality $\mathbf{m}%
_{n}(t)\allowbreak =\allowbreak \mathcal{A}_{n}(s,t)\mathbf{m}_{n}(s)$ and
then let us multiply from the right both sides of so obtained equality by $%
(X_{s}^{(n)}-m_{n}(s))^{T}.$ Finally let us take expectation of both sides.
We will get then:%
\begin{equation}
\Sigma _{n}(s,t)=\mathcal{A}_{n}(s,t)\Sigma _{n}(s,s).  \label{covar}
\end{equation}%
Let us remark that although relationship (\ref{covar}) has nicer intuitive
meaning it is less informative since $(0,0)$ entries of matrices $\Sigma
_{n}(s,t)$ are equal to zero.

Notice also that from the definition of matrices $\mathcal{A}_{n}(s,t),$
matrix $\mathcal{A}_{n}(s,t)$ is a submatrix of every matrix $\mathcal{A}%
_{k}(s,t)$ for $k\geq n.$ Consequently when $\mathcal{A}_{k}(s,t)$ is
non-singular then necessarily all matrices $\mathcal{A}_{n}(s,t)$ are
non-singular for $n\leq k.$

We have the following simple proposition.

\begin{proposition}
\label{CTLI}For $N-$MPR Markov process of degree $N$ for $s\leq t$

i) matrix $\mathcal{A}_{N}(s,t)$ is non-singular.

ii) $\mathcal{A}_{N}(s,u)\allowbreak =\allowbreak \mathcal{A}_{N}(t,u)%
\mathcal{A}_{N}(s,t)$ and $\mathcal{A}_{N}(s,s)\allowbreak =\allowbreak I-$%
identity matrix.

iii) There exists a family of non-singular, lower triangular, $(N+1)\times
(N+1)$ matrices $V_{N}(t),$ such that for all $s\leq u$ 
\begin{equation}
\mathcal{A}_{N}(s,u)=V_{N}(u)V_{N}^{-1}(s).  \label{fact}
\end{equation}

iv) Diagonal entries of matrices $V_{N}(s)$ are positive.
\end{proposition}

\begin{proof}
Proof is shifted to Section \ref{dowody}.
\end{proof}

Matrices $V_{N}(t)$ are not defined uniquely since we have for every
non-singular lower triangular matrix $F_{N}$ :%
\begin{equation*}
V_{N}(t)F_{N}(V_{N}(s)F_{N})^{-1}\allowbreak =\allowbreak
V_{N}(t)V_{N}^{-1}(s).
\end{equation*}%
Hence one can define equivalence relationship between matrices $V_{N}$
defined by (\ref{fact}) and call $V_{N}(t)$ and $V_{N}^{^{\prime }}(t)$
equivalent iff there exist non-singular lower triangular matrix $F_{N}$ with
all diagonal elements positive such that $V_{N}^{^{\prime }}(t)\allowbreak
=\allowbreak V_{N}(t)F_{N}.$

\begin{definition}
Every matrix $V_{N}(t)$ from this equivalence will be called structural
matrix of of the $N-$MPR process.
\end{definition}

\begin{corollary}
Matrix $V_{N}^{-1}(s)\mathbf{\Sigma }_{N}(s,t)\left( V_{N}^{T}(t)\right)
^{-1}$ is symmetric and matrix $\mathbf{C}_{N}(t,s)\mathbf{M}_{N}^{-1}(s)$
is lower-triangular.
\end{corollary}

\begin{proof}
Combining (\ref{podst}), (\ref{fact}) and the fact that matrix $\mathbf{%
\Sigma }_{n}(s,s)$ is symmetric we have:%
\begin{eqnarray*}
\mathbf{\Sigma }_{n}(s,t) &=&V_{n}(t)V_{n}^{-1}(s)\Sigma
_{n}(s,s),\allowbreak \\
\allowbreak \mathbf{\Sigma }_{n}^{T}(s,s) &=&\Sigma _{n}^{T}(t,s)\left(
V_{n}^{-1}(s)\right) ^{-1}V_{n}^{T}(t)
\end{eqnarray*}%
from which follows our assertion. Similarly multiplying both sides of (\ref%
{podst}) by $\mathbf{M}_{n}^{-1}(s)$ we see that $\mathbf{C}%
_{n}(t,s)M_{n}^{-1}(s)$ is equal to $\mathcal{A}_{n}(s,u)$ which is
lower-triangular.
\end{proof}

We have the following observation:

\begin{theorem}
\label{martingale}Process $\mathbf{X}$ is $N-$MPR process with structural
matrix $V_{N}(t)$ if and only if there exist $N$ (polynomial) martingales $(%
\mathcal{M}_{i}(s),\mathcal{F}_{\leq s})_{s\in I}$ , $i\allowbreak
=\allowbreak 1,\ldots ,N,$ such that each $\mathcal{M}_{i}(t)$ is a
polynomial in $X_{t}$ of degree $i$. Let vector $\mathbf{M}_{N}(t)$ with
entries $\mathcal{M}_{i}(t)$ be defined by $V_{N}^{-1}(t)\mathbf{X}%
_{t}^{(N)} $ then matrix $V_{N}(t)$ is the structural matrix of the process $%
\mathbf{X.} $
\end{theorem}

\begin{proof}
If $\mathbf{X}$ is $N-$MPR process then we have for $s<t$ 
\begin{equation*}
E(\mathbf{M}_{N}(t)|\mathcal{F}_{\leq s})\allowbreak =\allowbreak
V_{N}^{-1}(t)\mathcal{A}_{N}(s,t)X_{s}^{(N)}\allowbreak =\allowbreak
V_{N}^{-1}(t)V_{N}(t)V_{N}^{-1}(s)X_{s}^{(N)}\allowbreak =\allowbreak 
\mathbf{M}_{N}(s),
\end{equation*}%
by (\ref{fact}). Conversely if vector $\mathbf{M}_{N}(t)$ has entries being
polynomial martingales then there exist matrix $W_{N}(t)$ such that $\mathbf{%
M}_{N}(t)\allowbreak =\allowbreak W_{N}(t)\mathbf{X}_{t}^{(N)}.$ Then by the
martingale property of $\mathbf{M}_{n}(t)$ we have:%
\begin{equation*}
W_{N}(t)E(\mathbf{X}_{t}^{(N)}|\mathcal{F}_{\leq s})\allowbreak =\allowbreak
W_{N}(s)\mathbf{X}_{s}^{(N)}.
\end{equation*}%
So $E(\mathbf{X}_{t}^{(N)}|\mathcal{F}_{\leq s})\allowbreak =\allowbreak
W_{N}^{-1}(t)W_{N}(s)\mathbf{X}_{s}^{(N)}.$ Moreover matrix $%
W_{N}^{-1}(t)W_{N}(s)$ satisfies Proposition \ref{CTLI}, ii), so $\mathbf{X}$
is $N-$MPR.
\end{proof}

\begin{corollary}
If $(\mathcal{M}_{i}(s),\mathcal{F}_{\leq s})_{s\in I}$ , $i\allowbreak
=\allowbreak 1,\ldots ,n$ are polynomial martingales of some process $N-$MPR 
$\mathbf{X}$ then for every lower-triangular $n\times n$ matrix $%
F_{n}\allowbreak =\allowbreak \lbrack f_{i,j}]_{i,j=1,\ldots ,n},$ $%
(\sum_{k=0}^{j}f_{j,k}\mathcal{M}_{k}(s),\mathcal{F}_{\leq s})_{s\in I}$ , $%
j\allowbreak =\allowbreak 1,\ldots ,n$ are also martingales.
\end{corollary}

The following practical observation that can be useful if martingale
polynomials are known.

\begin{remark}
\label{struct}Structural matrix is equal to $[c_{ij}(t)]_{i,j=0,\ldots
,,n}^{-1}$ where $c_{i,j}(t)$ is a coefficient by $x^{j}$ of the $i-th$
polynomial martingale.
\end{remark}

We have also the following characterization of the processes with
independent increments within the class of MPR processes. In fact within the
class of MPR processes we can slightly generalize the notion of 'independent
increment property' and introduce the notion of independent increment
property of degree $N$ if for all $0<$ $n\leq N:$ $E((X_{t}-X_{s})^{n}|%
\mathcal{F}_{\leq s})$ is non-random i.e. is a.s. surely a constant.

\begin{proposition}
\label{indep}Let $\mathbf{X}$ be a $N-$MPR. Then $E((X_{t}-E(X_{t}|\mathcal{F%
}_{\leq s}))^{j}|\mathcal{F}_{\leq s})$ does not depend on $X_{s}$ for $%
j\allowbreak =\allowbreak 1,\ldots ,N$ iff for every $N\geq n>0$ and $t>s:$%
\begin{equation*}
E(X_{t}^{n}|\mathcal{F}_{\leq s})=\sum_{j=0}^{n}\binom{n}{j}\gamma
_{1,1}^{j}(s,t)X_{s}^{j}\sum_{k=0}^{n-j}\binom{n-j}{k}\gamma
_{n-j-k,0}(s,t)\gamma _{1,0}^{k}(s,t).
\end{equation*}
\end{proposition}

\begin{proof}
By our assumptions we have : $E(X_{t}|\mathcal{F}_{\leq s})\allowbreak
=\allowbreak \gamma _{1,1}(s,t)X_{s}\allowbreak +\allowbreak \gamma
_{1,0}(s,t).$ Further $E(X_{t}-E(X_{t}|\mathcal{F}_{\leq s}))^{n}|\mathcal{F}%
_{\leq s})$ being a polynomial in $X_{s}$ by our assumptions is a constant
polynomial. Conversely if this property holds for all $n\leq N$ then $%
X_{t}-E(X_{t}|\mathcal{F}_{\leq s})$ is independent on $\mathcal{F}_{\leq s}$
of degree $N.$ Hence let us $(E(X_{t}-E(X_{t}|\mathcal{F}_{\leq s}))^{n}|%
\mathcal{F}_{\leq s})$ denote $\gamma _{n,0}(s,t).$ We have:%
\begin{gather*}
E(X_{t}^{n}|\mathcal{F}_{\leq s})\allowbreak =\allowbreak E((X_{t}-E(X_{t}|%
\mathcal{F}_{\leq s})+\gamma _{1,1}X_{s}+\gamma _{1,0})^{n}|\mathcal{F}%
_{\leq s})= \\
\sum_{j=0}^{n}\binom{n}{j}\gamma _{n-j,0}(\gamma _{1,1}X_{s}+\gamma
_{1,0})^{j}\allowbreak =\allowbreak \sum_{j=0}^{n}\binom{n}{j}\gamma
_{n-j,0}\sum_{k=0}^{j}\binom{j}{k}\gamma _{1,1}^{k}X_{s}^{k}\gamma
_{1,0}^{j-k}= \\
\sum_{k=0}^{n}\binom{n}{k}\gamma _{1,1}^{k}X_{s}^{k}\sum_{j=k}^{n}\binom{n-k%
}{j-k}\gamma _{n-j,0}\gamma _{1,0}^{j-k}=\sum_{k=0}^{n}\binom{n}{k}\gamma
_{1,1}^{k}X_{s}^{k}\sum_{m=0}^{n-k}\binom{n-k}{m}\gamma _{n-m-k,0}\gamma
_{1,0}^{m}.
\end{gather*}
\end{proof}

As a corollary we get the following characterization of the processes with
independent increments.

\begin{corollary}
\label{ind_inc}Let $\mathbf{X}$ be a $N-$MPR. Then for $X_{t}-X_{s}$ to be
independent of $\mathcal{F}_{\leq s}$ of degree $N$ it is necessary and
sufficient that for every $N\allowbreak \geq \allowbreak n>\allowbreak 0$
and $t>s$ structural matrix $V_{N}(t)\allowbreak =\allowbreak \left[
v_{i,j}(t)\right] _{i,j=0,\ldots N}$ is of the form: 
\begin{equation*}
v_{i,j}(t)=\left\{ 
\begin{array}{ccc}
0 & if & i<j \\ 
1 & if & i=j \\ 
\binom{i}{j}g_{i-j}(t) & if & i>j%
\end{array}%
\right. .
\end{equation*}%
where $g_{k}$ is some continuous function of $t$ such that $%
g_{i}(0)\allowbreak =\allowbreak 0.$
\end{corollary}

\begin{proof}
Is shifted to Section \ref{dowody}.
\end{proof}

\begin{example}
\label{str_stac}Now let us assume that process $\mathbf{X}$ is stationary in
wider sense of degree $N$ i.e. we will assume that $\forall t\in
I,k=1,\ldots ,2N:EX_{t}^{k}=m_{k},$ $\limfunc{cov}(X_{t}^{j},X_{s}^{k})%
\allowbreak =\allowbreak c_{j,k}(\left\vert s-t\right\vert )$ for all $%
j,k\allowbreak =\allowbreak 0,\ldots ,N$ , where $c_{j,k}$ are some
functions of one variable only. Following arguments used in the theory of
(weakly) stationary processes one can easily prove that functions $%
c_{j,j}(.) $ must be positive definite. (\ref{podst}) now takes the form 
\begin{equation*}
\left[ c_{j,k}(t-s)\right] _{j,k=0,\ldots ,N}=\mathcal{A}_{N}(s,t)\mathbf{M}%
_{N},
\end{equation*}%
from which it follows that $\mathcal{A}_{N}(s,t)\allowbreak =\allowbreak
A_{N}(t-s)$ for some matrix $A_{n}$ with entries depending on $t.$ Now
Proposition \ref{CTLI},ii) leads to the conclusion that 
\begin{equation*}
A_{N}(t)A_{N}(s)\allowbreak =A_{n}(t+s).
\end{equation*}%
that is that matrices $A_{N}(t)$ and $A_{N}(s)\allowbreak $ commute. Hence
in particular eigenvectors of $A_{N}(t)$ do not depend on $t.$ Another words 
\begin{equation*}
A_{N}(t)=G_{N}\Lambda _{N}(t)G_{N}^{-1},
\end{equation*}%
where $G_{N}$ is some lower triangular matrix and $\Lambda _{N}(t)$ is some
diagonal matrix with entries depending on $t.$ Now one can see that since
entries of $\Lambda _{N}$ are continuous and they have to satisfy
multiplicative Cauchy equation and we deduce that $\Lambda
_{N}(t)\allowbreak =\allowbreak diag\{1,\exp \left( \alpha _{1}t\right)
,\ldots ,\exp \left( \alpha _{N}t\right) \}$ for some constants $\alpha
_{i}. $ There exist $N$ (polynomial ) martingales $of$ the form $\exp
(-\alpha _{i}t)Z_{i}(X_{t}),$ where $Z_{i}(x)$ are some polynomials of
degree at most $i,$ $i\allowbreak =\allowbreak 1,\ldots ,N.$
\end{example}

\begin{remark}
Notice that one cannot deduce that for the process analyzed in the last
example there exists a reversed martingale. We could deduce this if the
stationarity assumption would hold for all $n.$ Then since we deal with the
distributions determined by moments we would deduce that transitional
distribution depends on the time distance and one dimensional distribution
does not depend on $t.$ Hence we deal with real (strong) stationarity.
\end{remark}

\subsection{Important example ($q-$Wiener and $(\protect\alpha ,q)-$OU
processes)\label{example}}

In this subsection we will discuss an example of a family of Markov
processes indexed by a parameter $q\in (-1,1].$ For all values of $%
\left\vert q\right\vert <1$ the processes described by this family do not
have independent increments. For $q\allowbreak =\allowbreak 1$ we deal with
the classical Wiener process (that has independent increments). For $%
q\allowbreak =\allowbreak 0$ we deal with the so called 'free' Wiener
process a process whose $1-$dimensional distributions have density equal to $%
\frac{1}{2\pi }\sqrt{4-x^{2}},$ i.e. has Wigner distribution. We will
consider also the other, related family of stationary, Markov processes
namely the so called $(\alpha ,q)-$Ornstein--Uhlenbeck processes obtained
from $q-$Wiener processes by suitable time transformation. Again if $%
q\allowbreak =\allowbreak 1$ $(\alpha ,q)-$Ornstein--Uhlenbeck would become
ordinary $\alpha -$Ornstein--Uhlenbeck process.

The processes that we are going to present were first defined and described
in \cite{Bo} as a side result of some non-commutative probability
considerations. In \cite{Bryc2001M}, \cite{Bryc2001S} the discrete version
of $(\alpha ,q)-$OU process appeared and in \cite{BryMaWe07} $q-$Wiener
process appeared again as a side result of considerations concerning
quadratic harnesses. More recently in more detail these processes were
described in \cite{Szab-OU-W}. Among other properties both these processes
do not have continuos paths as indicated in \cite{Szab-OU-W}. Consequently
they are not diffusion processes.

To describe their properties briefly and also to illustrate assertions of
Theorem \ref{martingale} we must recall some facts concerning $q-$Hermite
and Al-Salam--Chihara polynomials and densities with respect to which they
are orthogonal. To do it swiftly let us introduce the following notation.

Assume that $-1<q\leq 1.$ We will use traditional notation of the $q-$series
theory i.e. $\left[ 0\right] _{q}\allowbreak =\allowbreak 0,$ $\left[ n%
\right] _{q}\allowbreak =\allowbreak 1+q+\ldots +q^{n-1}\allowbreak ,$ $%
\left[ n\right] _{q}!\allowbreak =\allowbreak \prod_{i=1}^{n}\left[ i\right]
_{q},$ with $\left[ 0\right] _{q}!\allowbreak =1,\QATOPD[ ] {n}{k}%
_{q}\allowbreak =\allowbreak \left\{ 
\begin{array}{ccc}
\frac{\left[ n\right] _{q}!}{\left[ n-k\right] _{q}!\left[ k\right] _{q}!} & 
, & n\geq k\geq 0, \\ 
0 & , & otherwise.%
\end{array}%
\right. $

It will be also helpful to use the so called $q-$Pochhammer symbol defined
for $n\geq 1$ by$:\left( a;q\right) _{n}=\prod_{i=0}^{n-1}\left(
1-aq^{i}\right) ,$ with $\left( a;q\right) _{0}=1$ , $\left(
a_{1},a_{2},\ldots ,a_{k};q\right) _{n}\allowbreak =\allowbreak
\prod_{i=1}^{k}\left( a_{i};q\right) _{n}$.

Often $\left( a;q\right) _{n}$ as well as $\left( a_{1},a_{2},\ldots
,a_{k};q\right) _{n}$ will be abbreviated to $\left( a\right) _{n}$ and $%
\left( a_{1},a_{2},\ldots ,a_{k}\right) _{n},$ if it will not cause
misunderstanding.

In particular it is easy to notice that $\left( q\right) _{n}=\left(
1-q\right) ^{n}\left[ n\right] _{q}!$ and that\newline
$\QATOPD[ ] {n}{k}_{q}\allowbreak =$\allowbreak $\allowbreak \left\{ 
\begin{array}{ccc}
\frac{\left( q\right) _{n}}{\left( q\right) _{n-k}\left( q\right) _{k}} & ,
& n\geq k\geq 0, \\ 
0 & , & otherwise.%
\end{array}%
\right. $ \newline
Let us remark that $\left[ n\right] _{1}\allowbreak =\allowbreak n,\left[ n%
\right] _{1}!\allowbreak =\allowbreak n!,$ $\QATOPD[ ] {n}{k}_{1}\allowbreak
=\allowbreak \binom{n}{k},$ $\left( a;1\right) _{n}\allowbreak =\allowbreak
\left( 1-a\right) ^{n}$ and $\left[ n\right] _{0}\allowbreak =\allowbreak 
\begin{array}{ccc}
1 & if & n\geq 1, \\ 
0 & if & n=0,%
\end{array}%
,$ $\left[ n\right] _{0}!\allowbreak =\allowbreak 1,$ $\QATOPD[ ] {n}{k}%
_{0}\allowbreak =\allowbreak 1,$ $\left( a;0\right) _{n}\allowbreak
=\allowbreak \left\{ 
\begin{array}{ccc}
1 & if & n=0, \\ 
1-a & if & n\geq 1.%
\end{array}%
\right. .$ Now let us introduce the $q-$Hermite $\left\{ H_{n}\left(
x|q\right) \right\} _{n\geq -1}$ polynomials. They satisfy the following $3-$%
term recurrence%
\begin{equation}
H_{n+1}\left( x|q\right) \allowbreak =\allowbreak xH_{n}\left( x|q\right) - 
\left[ n\right] _{q}H_{n-1}\left( x\right) ,  \label{He}
\end{equation}%
with $H_{-1}\left( x|q\right) \allowbreak =\allowbreak 0$, $H_{1}\left(
x|q\right) \allowbreak =\allowbreak 1$. Let us also introduce the so called
Al-Salam--Chihara $\left\{ P_{n}\right\} _{n\geq -1}$ polynomials. Those are
polynomials satisfying the following $3-$term recurrence: 
\begin{equation}
P_{n+1}(x|y,\rho ,q)=(x-\rho yq^{n})P_{n}(x|y,\rho ,q)-(1-\rho
^{2}q^{n-1})[n]_{q}P_{n-1}(x|y,\rho ,q),  \label{AlSC}
\end{equation}%
with $P_{-1}\left( x|y,\rho ,q\right) \allowbreak =\allowbreak 0,$ $%
P_{0}\left( x|y,\rho ,q\right) \allowbreak =\allowbreak 1$. The polynomials $%
\left\{ P_{n}\right\} $ have nice probabilistic interpretation see e.g. \cite%
{bms}.

Let us only remark supporting intuition that $H_{n}(x|1)\allowbreak
=\allowbreak He_{n}(x),$ where $He_{n}$ denotes so called 'probabilistic'
Hermite polynomial i.e. one orthogonal with respect to the measure with the
density $\exp (-x^{2}/2)/\sqrt{2\pi }.$ Similarly $P_{n}(x|y,\rho
,1)\allowbreak =\allowbreak He_{n}(\frac{x-\rho y}{\sqrt{1-\rho _{2}}}%
)(1-\rho ^{2})^{n/2}.$ For $q\allowbreak =\allowbreak 0$ we notice that $%
H_{n}(x|0)\allowbreak =\allowbreak U_{n}(x/2)$ where $U_{n}$ is the
Chebyshev polynomial of the second kind. One can show (see e.g. \cite%
{SzablAW}) that $P_{n}(x|y,\rho ,0)\allowbreak =\allowbreak
U_{n}(x/2)\allowbreak -\allowbreak \rho yU_{n-1}(x/2).$

It is also known that polynomials $H_{n}$ and $P_{n}$ are respectively
orthogonal with respect to the following measures with the densities:%
\begin{eqnarray*}
f_{N}\left( x|q\right) &=&\frac{\sqrt{1-q}\left( q\right) _{\infty }}{2\pi 
\sqrt{4-(1-q)x^{2}}}\prod_{k=0}^{\infty }\left(
(1+q^{k})^{2}-(1-q)x^{2}q^{k}\right) I_{S\left( q\right) }\left( x\right) ,
\\
f_{CN}\left( x|y,\rho ,q\right) &=&f_{N}\left( x|q\right)
\prod_{k=0}^{\infty }\frac{(1-\rho ^{2}q^{k})}{w_{k}\left( x,y|\rho
,q\right) }I_{S\left( q\right) }\left( x\right) ,
\end{eqnarray*}%
where $S(q)\allowbreak =\allowbreak \lbrack -2/\sqrt{1-q},2/\sqrt{1-q}]$ if $%
\left\vert q\right\vert <1$ and $\mathbb{R}$ for $q\allowbreak =\allowbreak
1,$ $w_{k}\left( x,y|\rho ,q\right) \allowbreak =\allowbreak (1-\rho
^{2}q^{2k})^{2}-\rho q^{k}(1+\rho ^{2}q^{2k})xy+\rho
^{2}q^{2k}(x^{2}+y^{2}). $ Compare \cite{Biane98} for $q\allowbreak
=\allowbreak 0.$ \newline
We have 
\begin{eqnarray*}
\int_{S(q)}H_{m}(x|q)H_{n}(x|q)f_{N}(x|q)dx &=&\delta _{n,m}[n]_{q}!, \\
\int_{S(q)}P_{m}(x|,y,\rho ,q)P_{n}(x|y,\rho ,q)f_{CN}(x|y,\rho ,q)dx
&=&\delta _{n,m}[n]_{q}!(\rho ^{2})_{n}.
\end{eqnarray*}

Now $(\alpha ,q)-$OU process is a stationary Markov process with $f_{N}(.|q)$
as its stationary distribution and $f_{CN}(.|y,\exp (-\alpha (t-s),q)$ as
its density of transition distribution i.e. $P(X_{t}\in A|X_{s}\allowbreak
=\allowbreak y)\allowbreak =\allowbreak \int_{A\cap S\left( q\right)
}f_{CN}(x|y,\exp (-\alpha (t-s),q)dx.$

$q-$Wiener process is obtained from $(\alpha ,q)-$OU process by time
transformation. More precisely let $\mathbf{Y}$ be given $(\alpha ,q)-$OU
process. Let us define:%
\begin{equation}
X_{0}=0;~\forall \tau >0:X_{\tau }=\sqrt{\tau }Y_{\log \tau /2\alpha }.
\label{def_q-W}
\end{equation}%
Process $\mathbf{X\allowbreak =}\allowbreak \left( X_{\tau }\right) _{\tau
\geq 0}$ will be called $q-$Wiener process. To see examples of matrices
structural matrices $V_{n}(t)\mathcal{\ }$of these processes that are
defined by these processes let us recall some of their properties.

Let $\mathbf{Y}$ and $\mathbf{X\allowbreak }$ be $(\alpha ,q)-$OU and $q-$%
Wiener processes respectively. For $\forall n\geq 1,$ $t,s\in \mathbb{R},,$ $%
\tau >\sigma \geq 0$ and $0<q\leq 1$ we have almost surely:%
\begin{eqnarray}
E(H_{n}(Y_{t}|q)|\mathcal{F}_{\leq s})\allowbreak &=&\allowbreak e^{-n\alpha
(t-s)}H_{n}(Y_{s}|q),  \label{OU} \\
\tau ^{n/2}E(H_{n}(X_{\tau }/\sqrt{\tau }|q)|\mathcal{F}_{\leq \sigma })
&=&\sigma ^{n/2}H_{n}(X_{\sigma }/\sqrt{\sigma }|q).  \label{q-W}
\end{eqnarray}
Now using Remark \ref{struct} and remembering that first $5$ $q-$Hermite
polynomials have the form: $H_{1}(x|q)\allowbreak =\allowbreak x,$ $%
H_{2}(x|q)\allowbreak =\allowbreak x^{2}-1,$ $H_{3}(x|q)\allowbreak
=\allowbreak x^{3}-(2+q)x,$ $H_{4}(x|q)\allowbreak =\allowbreak
x^{4}-(q^{2}+2q+3)x^{2}+[3]_{q},$ $H_{5}(x|q)\allowbreak =\allowbreak
x^{5}\allowbreak -\allowbreak x^{3}(q^{3}+2q^{2}+3q+4)\allowbreak
+\allowbreak x(q^{4}+3q^{3}+4^{2}+4q+3)$ we get the following structure
matrices of these two processes.

$(\alpha ,q)-$OU process has the following structural matrix $\allowbreak $%
\newline
$V_{6}(t|q)\allowbreak =\allowbreak \exp (tA_{6})\left[ 
\begin{array}{cccccc}
1 & 0 & 0 & 0 & 0 & 0 \\ 
0 & 1 & 0 & 0 & 0 & 0 \\ 
1 & 0 & 1 & 0 & 0 & 0 \\ 
0 & 2+q & 0 & 1 & 0 & 0 \\ 
2+q & 0 & 3+2q+q^{2} & 0 & 1 & 0 \\ 
0 & 5+6q+3q^{2}+q^{3}) & 0 & 4+3q+2q^{2}+q^{3} & 0 & 1%
\end{array}%
\right] ,$ where $A_{6}\allowbreak =\allowbreak diag\{0,\alpha ,\ldots
,5\alpha \}$. Notice that for $q\allowbreak =\allowbreak 1$ we obtain the
structural matrix of the ordinary Ornstein--Uhlenbeck process. This matrix
can be defined for any $n.$ Namely since by Remark \ref{struct} one has to
take coefficients of the probabilistic Hermite polynomials. Compose a lower
triangular matrix of them. Finally invert this matrix. Since 
\begin{equation*}
H_{n}(x|1)\allowbreak =\allowbreak \sum_{j=0}^{\left\lfloor n.2\right\rfloor
}(-1)^{j}x^{n-2j}\binom{n}{2j}(2j-1)!!,
\end{equation*}%
we obtain $V_{n}\allowbreak =\allowbreak \lbrack v_{ij}]_{i,j=0,\ldots ,n}$
where $v_{ij}\allowbreak =\allowbreak \left\{ 
\begin{array}{ccc}
0 & if & j>i,i-j\text{ is odd,} \\ 
1 & if & i=j, \\ 
\binom{i}{j}(2(i-j)-1)!! & if & i-j>0\text{ is even.}%
\end{array}%
\right. .$ For $q\allowbreak =\allowbreak 0$ we perform similar operation
and get:%
\begin{equation*}
U_{n}(x/2)\allowbreak =\allowbreak \sum_{j=0}^{\left\lfloor n/2\right\rfloor
}(-1)^{j}x^{n-2j}\binom{n}{j}\frac{n-2j+1}{n-j+1},
\end{equation*}%
and $V_{n}\allowbreak =[v_{ij}]_{i,j=0,\ldots ,n}\allowbreak $where $%
v_{ij}\allowbreak =\allowbreak \left\{ 
\begin{array}{ccc}
0 & if & j>i,i-j\text{ is odd,} \\ 
1 & if & i=j, \\ 
\binom{i}{(i-j)/2}\frac{2(j+1)}{i+j+2} & if & i-j>0\text{ is even.}%
\end{array}%
\right. $ As one can easily see $\binom{i}{(i-j)/2}\frac{2(j+1)}{i+j+2}/%
\binom{i}{j}$ does nor depend on $i-j$ hence 'free' Wiener process do not
have independent increments.

For the $q-$Wiener process $V_{6}(t|q)$ is given by the following formula:%
\newline
$\left[ 
\begin{array}{cccccc}
1 & 0 & 0 & 0 & 0 & 0 \\ 
0 & 1 & 0 & 0 & 0 & 0 \\ 
t & 0 & 1 & 0 & 0 & 0 \\ 
0 & t(2+q) & 0 & 1 & 0 & 0 \\ 
t^{2}(2+q) & 0 & t(3+2q+q^{2}) & 0 & 1 & 0 \\ 
0 & t^{2}(5+6q+3q^{2}+q^{3}) & 0 & t(4+3q+2q^{2}+q^{3}) & 0 & 1%
\end{array}%
\right] .$ Notice that entries $v_{i,i-2}$ of this matrix are not of the
form $\binom{i}{i-2}g_{2}(t)$ proving that $q-$Wiener process does not have
independent increments. By setting $q\allowbreak =\allowbreak 1$ we get
structural matrix of classical Brownian motion which has independent
increments.

\section{Markov processes with orthogonal martingales\label{ort_mar}}

Now let us assume that our process $\mathbf{X}$ is MPR and there exist a
family of polynomial martingales that additionally are orthogonal with
respect to one dimensional distribution. More precisely let us assume that $%
\forall t\in I$ there exist a family of polynomials $\left\{
p_{n}(.;t)\right\} _{n\geq 0}$ such that $p_{n}(x;t)$ is a polynomial of
degree $n$ in $x$ and moreover that 
\begin{eqnarray}
Ep_{n}(X_{t};t)p_{m}(X_{t};t) &=&\hat{p}_{n}(t)\delta _{n,m},  \label{ort} \\
E(p_{n}(X_{t};t)|\mathcal{F}_{\leq s}) &=&p_{n}(X_{s};s)~~a.s.  \label{Mar}
\end{eqnarray}%
Recall that such process possessing sequence of orthogonal, polynomial
martingales i.e. MPR processes such that conditions (\ref{ort}) and (\ref%
{Mar}) are satisfied will be called for brevity OPM class (of Markov
processes).

To express these conditions in terms of the notions that we have introduced
above let us assume that our process has family of structural matrices $%
\left\{ V_{n}(t)\right\} _{n\geq 1,t\in I}.$ We have the following simple
observation:

\begin{proposition}
\label{wiel_ort}Assuming that $\mathbf{X}$ is OPM process with structural
matrix $\left\{ V_{n}(t)\right\} _{n\geq 1,t\in I}.$ A necessary and
sufficient for the existence of the family of orthogonal polynomial
martingales $\left\{ p_{n}(.;t)\right\} _{n\geq 0}$ is the existence of a
family of diagonal matrices $J_{n}(t)$ such that $\forall t\in I,$ $n\geq 1:$
\begin{equation}
V_{n}(t)=D_{n}(t)J_{n}(t)  \label{mart1}
\end{equation}%
where matrix $D_{n}(t)$ is defined by the Cholesky decomposition of the
moment matrix i.e. by $\mathbf{M}_{n}(t)\allowbreak =\allowbreak
D_{n}(t)D_{n}^{T}(t)$.
\end{proposition}

\begin{proof}
Now $\mathbf{p}_{n}(X_{t};t)\allowbreak =\allowbreak
V_{n}^{-1}(t)X_{t}^{(n)} $ is a orthogonal polynomial martingale by Theorem %
\ref{martingale}. Further we have $E\mathbf{p}_{n}(X_{t};t)\mathbf{p}%
_{n}^{T}(X_{t};t)\allowbreak =\allowbreak V_{n}^{-1}(t)\mathbf{M}%
_{n}(t)(V_{n}^{-1}(t))^{T}$ by orthogonal assumption. Also by the assumption
the last matrix is diagonal. Since matrix $\mathbf{M}_{n}$ is positive
definite we deduce that it has Cholesky decomposition $\mathbf{M}%
_{n}(t)\allowbreak =\allowbreak D_{n}(t)D_{n}^{T}(t)$ with $D_{n}$ being
lower-triangular. Consequently we have $%
V_{n}^{-1}(t)D_{n}(t)(V_{n}^{-1}(t)D_{n}(t))^{T}\allowbreak =\allowbreak 
\mathbf{P}_{n}(t)$ where $\mathbf{P}_{n}$ denotes some diagonal matrix with
positive entries. Since Cholesky decomposition is unique we deduce that $%
V_{n}^{-1}(t)D_{n}(t)\allowbreak =\allowbreak \sqrt{\mathbf{P}_{n}(t)}$
where $\sqrt{\mathbf{P}_{n}(t)}$ denotes diagonal matrix with entries that
are square roots of entries of matrix $\mathbf{P}_{n}(t).$ Consequently $%
V_{n}(t)\allowbreak =\allowbreak D_{n}(t)\left( \sqrt{\mathbf{P}_{n}(t)}%
\right) ^{-1}$.

Conversely let us assume that structural matrix of our process is of the
form 
\begin{equation*}
V_{n}(t)\allowbreak =\allowbreak D_{n}(t)J_{n}(t)
\end{equation*}%
where $D_{n}(t)$ is the Cholesky decomposition matrix of the moment matrix $%
\mathbf{M}_{n}(t)$ and $J_{n}(t)$ is a diagonal matrix with positive
entries. Then by assumption $\mathbf{p}_{n}(X_{t};t)\allowbreak =\allowbreak
V_{n}^{-1}(t)X_{t}^{(n)}$ is a martingale and we have 
\begin{eqnarray*}
E\mathbf{p}_{n}(X_{t};t)\mathbf{p}_{n}^{T}(X_{t};t)\allowbreak
&=&J_{n}^{-1}(t)D_{n}^{-1}(t)\mathbf{M}_{n}(t)\left( D_{n}^{-1}(t)\right)
^{T}J_{n}^{-1}(t) \\
&=&J_{n}^{-2}(t)
\end{eqnarray*}%
by the properties of Cholesky decomposition. Another words entries of the
vector $\mathbf{p}_{n}(X_{t};t)$ are mutually orthogonal.
\end{proof}

\begin{example}
To see examples of such processes let us recall the example with strongly
stationary processes, i.e. Example \ref{str_stac}. As we know structural
matrix of stationary process must be of the form $V_{n}(t)\allowbreak
=\allowbreak G_{n}R_{n}(t)G_{n}^{-1}$ for some lower-triangular matrix $%
G_{n} $ and a diagonal matrix $R_{n}$ of the form $\exp (t\Delta _{n})$,
where $\Delta _{n}$ is a diagonal matrix. Now since our process is assumed
to be stationary its one-dimensional distribution does not depend on $t$ and
consequently moment matrix does not depend on $t$ so matrices $D_{n}$ in its
Cholesky decomposition does not depend on $t.$ Consequently condition (\ref%
{mart1}) can be reduced to the existence of the family of diagonal matrices $%
J_{n}(t)$ and lower triangular matrix $F_{n}$ satisfying:%
\begin{equation*}
G_{n}R_{n}(t)G_{n}^{-1}F_{n}=D_{n}J_{n}(t).
\end{equation*}%
One of see that the best choice to select $F_{n}$ is to set $%
F_{n}\allowbreak =\allowbreak G_{n}I_{n}$ for some diagonal $I_{n}.$ Then
one could select matrix $J_{n}(t)\allowbreak =\allowbreak \mathbf{P}%
_{n}^{-1}(t)I_{n}$ and then $G_{n}\allowbreak =\allowbreak $ $D_{n}.$
Another words the structural matrix of the stationary process must be of the
form%
\begin{equation*}
D_{n}\exp (t\Delta _{n}).
\end{equation*}%
Now recall (see e.g. \cite{SzabChol}) that polynomials of the form $%
D_{n}^{-1}X^{(n)}$ are orthonormal and their moment matrix is $%
D_{n}D_{n}^{T} $. That is orthogonal martingales must be of the form 
\begin{equation*}
\mathcal{M}_{j}(t)=\kappa _{j}\exp (t\delta _{j})p_{j}(X_{j}).
\end{equation*}%
where $\Delta _{j},$ $\kappa _{j}$ are nonnegative reals., and $p_{j}(x)$
denotes polynomial of order $j$ orthogonal with respect to marginal
one-dimensional measure.

One can notice that $(\alpha ,q)-$OU process analyzed in Subsection \ref%
{example} is an example of Markov process having orthogonal polynomial
martingales (compare (\ref{OU})).
\end{example}

Hence let us concentrate on OMP process for which there exist a family of
orthogonal polynomial martingales $p_{n}(X_{t};t)$ i.e. for every $n$ and $%
s<t$ conditions (\ref{ort}) and (\ref{Mar}) are satisfied. To proceed
further let us assume that both one-dimensional and transitional
distributions of our process have the same support and moreover transitional
distribution is absolutely continuous with respect to the marginal one.

Returning to notation from the beginning of the paper let $\mu (.,t),$ $t\in
I$ denote one-dimensional distribution and let $\eta (.,t;y,s),$ $t,s\in I,$ 
$t>s,~y\in \limfunc{supp}(\mu (.,s))$. In terms of polynomials $p_{n}$ we
have for $\forall n\geq 1,$ $s<t$%
\begin{eqnarray}
\int p_{n}(x;t)\mu (dx,t) &=&0,  \label{Ort} \\
\int p_{n}(x;t)\eta (dx,t;y,s)dt &=&p_{n}(y;s)~\text{~a.s.}  \label{Mart}
\end{eqnarray}

Let us also introduce functions $\hat{p}_{n}(t)$ defined 
\begin{equation*}
Ep_{n}^{2}(X_{t};t)\allowbreak =\allowbreak \hat{p}_{n}(t).
\end{equation*}%
Recalling theory of martingales we see that functions $\hat{p}_{n}(t)$ are
expectations of the so called square brackets of martingales $%
p_{n}(X_{t};t). $ We need only one property of them namely that functions $%
\hat{p}_{n}(t)$ are non-decreasing and of course positive.

Further, unless otherwise stated we will assume the following condition to
be satisfied by these densities:

\begin{condition}
\label{sq_int}Assume that $\forall t,s\in I,~t>s,$ $y\in \limfunc{supp}\mu
(.,s):$ $\eta (.,t;y,s)<<\mu (.,t)$ and let $\phi (x,t;y,s)\allowbreak $
denote Radon--Nikodym derivative $\frac{d\eta }{d\mu }$. Suppose 
\begin{equation*}
\int (\phi (x,t;y,s))^{2}\mu (dx,t)<\infty ,
\end{equation*}%
for all $s<t$ and $y\in \limfunc{supp}\mu (.,s)$.
\end{condition}

\begin{theorem}
\label{expansion} Assume Condition \ref{sq_int}, then $\forall t,s\in
I,~t>s: $

\begin{equation}
\eta (dx,t;y,s)\allowbreak =\allowbreak \mu (dx,t)\sum_{n\geq 0}\frac{1}{%
\hat{p}_{n}(t)}p_{n}(x,t)p_{n}(y,s).  \label{trans}
\end{equation}%
The convergence is in $L_{2}(t).$ Moreover almost everywhere $\mu (.,s)$ we
have: \newline
$\sum_{n\geq 0}p_{n}^{2}(y,s)/\hat{p}_{n}(t)<\infty .$
\end{theorem}

\begin{proof}
First of all let us denote by $\left\{ r_{n}(x,t,y,s)\right\} _{n\geq -1}$
the family of polynomials that are orthogonal with respect to the measure $%
\eta (.,t;y,s)$. Let us denote by $\left\{ \gamma _{k,n}\right\} $ the
family of 'connection coefficients' of polynomials $\left\{
p_{n}(x,t)\right\} $ in $\left\{ r_{n}(x,t,y,s\right\} ,$ i.e. 
\begin{equation*}
p_{n}(x,t)=\sum_{k=0}^{n}\gamma _{k,n}r_{k}(x,t,y,s).
\end{equation*}%
By assumption we have $\eta (dx,t;y,s)\allowbreak =\allowbreak \phi \left(
x,t,y,s\right) \mu \left( dx,t\right) .$ we have also $\forall n\geq 1:\int
r_{n}(x,t,y,s)\eta (dx,t;y,s)\allowbreak =\allowbreak 0$ and $\int
p_{n}(x,t)\eta (dx,t;y,s)\allowbreak =\allowbreak p_{n}(y,s)$ we deduce that 
$\gamma _{0,n}\allowbreak =\allowbreak p_{n}(y,s)$. Hence following the idea
of generalization of expansion of the ratio of two densities (Radon--Nikodym
derivative in the case of continuous measures) presented in \cite{Szabl-exp}
we deduce that (\ref{trans}) is true. The other statement follows Bessel
inequality.
\end{proof}

\begin{example}
To see that Condition \ref{sq_int} is not satisfied by an empty set let us
return to Subsection \ref{example}. As shown in \cite{SzablKer}(Lemma 1,(v))
ratio of $f_{CN}/f_{N}$ is bounded and cut away from zero on $S(q)$ hence
square integrable. Hence at least for $q-$Wiener and $(\alpha ,q)-$OU
processes Condition \ref{sq_int} is satisfied. Besides let we have $\hat{p}%
_{n}(t)\allowbreak =\allowbreak t^{n}[n]_{q}!$ for the $q-$Wiener process
and $\hat{p}_{n}(t)\allowbreak =\allowbreak \exp (2n\alpha t)\left[ n\right]
_{q}!$ for the $(\alpha ,q)-$OU processes.
\end{example}

\begin{remark}
If we add assumption that $\int \int (\phi (x,t;y,s))^{2}\mu (dx,t)d\mu
\left( dy,s\right) <\infty $ and follow the fact that $\eta (dx,t;y,s)\mu
\left( dy,s\right) $ is the joint distribution of $(X_{s},X_{t})$ we can
notice that then (\ref{trans}) can be written as 
\begin{equation*}
\eta (dx,t;y,s)\mu \left( dy,s\right) \allowbreak =\allowbreak \mu \left(
dx,t\right) \mu \left( dy,t\right) \sum_{n\geq 0}\sqrt{\frac{\hat{p}_{n}(s)}{%
\hat{p}_{n}(t)}}\frac{p_{n}(x,t)}{\sqrt{\hat{p}_{n}(t)}}\frac{p_{n}(y,s)}{%
\sqrt{\hat{p}_{n}(s)}},
\end{equation*}%
which is nothing else but Lancaster type expansion (compare \cite%
{Lancaster58}, \cite{Lancaster63(1)}, \cite{Lancaster63(2)}, \cite%
{Eagleson64}, \cite{Lancaster75}) of $\eta (dx,t;y,s)\mu \left( dy,s\right) $
since by the definition of $\hat{p}_{n}(t)$ polynomials $\left\{ \frac{%
p_{n}(x,t)}{\sqrt{\hat{p}_{n}(t)}}\right\} $ are orthonormal. As a side
result we have then $\sum_{n\geq 1}\hat{p}_{n}(s)/\hat{p}_{n}(t)<\infty .$
Following \cite{Eagleson64}, \cite{Lancaster75} we know that not only $q-$%
Wiener and $(\alpha ,q)-$OU processes allow expansion (\ref{trans}) but also
processes with marginals that belong to Meixner classes with infinite
support.
\end{remark}

We have immediate corollary.

\begin{corollary}
Under assumptions of Theorem \ref{expansion} for $n\geq 1,s<t$ we have :%
\begin{equation*}
E(p_{n}(X_{s};s)|\mathcal{F}_{\geq t})=\frac{\hat{p}_{n}(s)}{\hat{p}_{n}(t)}%
p_{n}(X_{t};t).
\end{equation*}%
Consequently $\left\{ p_{n}(X_{t};t)/\hat{p}_{n}(t);\mathcal{F}_{\geq
t}\right\} _{t\in I}$ is the reversed (polynomial) martingale.
\end{corollary}

\begin{proof}
First notice that $\allowbreak \phi \left( x,t,y,s\right) \mu (dy,t)\mu
(dx,s)$ is a joint distribution of $(X_{s},X_{t}),$ consequently $\eta
(dy,s;x,t)\allowbreak =\allowbreak \mu (dy,s)\sum_{n\geq 0}\frac{1}{\hat{p}%
_{n}(t)}p_{n}(x,t)p_{n}(y,s)$ is the conditional distribution of $%
X_{s}|X_{t}=x$. Hence we have \newline
$E(p_{m}(X_{s};s)|X_{t}\allowbreak =\allowbreak x)\allowbreak =\allowbreak
\int p_{m}(y;s)\mu (dy.,s)(\sum_{n\geq 0}\frac{1}{\hat{p}_{n}(t)}%
p_{n}(x,t)p_{n}(y,s))\allowbreak =\allowbreak \frac{\hat{p}_{m}(s)}{\hat{p}%
_{m}(t)}p_{m}(X_{t};t)$ a.s. .
\end{proof}

Let us assume that polynomials $\left\{ p_{n}(x;t)\right\} _{n\geq -1}$
satisfy the following 3-term recurrence:%
\begin{equation}
xp_{n}(x;t)=\alpha _{n+1}(t)p_{n+1}(x;t)+\beta _{n}(t)p_{n}(x;t)+\gamma
_{n-1}(t)p_{n-1}(x;t),  \label{w_o}
\end{equation}%
with $p_{-1}(x;t)\allowbreak =\allowbreak 0,$ $p_{0}(x;t)\allowbreak
=\allowbreak 1$. We have the following immediate observations:

\begin{proposition}
\label{poczatek}i) $x\allowbreak =\allowbreak \alpha _{1}(t)p_{1}(x;t)+\beta
_{0}(t),$ hence $EX_{t}\allowbreak =\allowbreak \beta _{0}(t),$ $%
p_{1}(x;t)=(x-\beta _{0}(t))/\alpha _{1}(t)$

ii) $EX_{t}^{2}\allowbreak =\allowbreak \alpha _{1}^{2}(t)\hat{p}(t)+\beta
_{0}^{2}(t),$ consequently $\limfunc{var}(X_{t})\allowbreak =\allowbreak
\alpha _{1}^{2}(t)\hat{p}_{1}\left( t\right) .$

iii) $\gamma _{n-1}(t)\hat{p}_{n-1}(t)=\alpha _{n}(t)\hat{p}_{n}(t)$ in
particular $\gamma _{0}(t)\allowbreak =\allowbreak \alpha _{1}(t)\hat{p}(t)$%
. Since $\forall t\in I,~n\geq 1:\hat{p}_{n}(t)>0$ we deduce that $\forall
t\in I,~n\geq 1:\alpha _{n}(t)\gamma _{n-1}(t)>0.$

iv) $xp_{1}(x;t)\allowbreak =\allowbreak \alpha _{2}(t)p_{2}(x;t)+\beta
_{1}(t)p_{1}(x;t)+\gamma _{0}(t)$ hence 
\begin{eqnarray*}
E(X_{t}p_{1}(X_{t};t))\allowbreak &=&\allowbreak \gamma _{0}(t), \\
p_{2}(x;t)\allowbreak &=&\allowbreak \frac{(x-\beta _{0}(t))(x-\beta
_{1}(t))-\gamma _{0}(t)\alpha _{1}(t)}{\alpha _{1}(t)\alpha _{2}(t)},
\end{eqnarray*}%
v) $\alpha _{1}(t)p_{1}^{2}(x;t)\allowbreak =\allowbreak \alpha
_{2}(t)p_{2}(x;t)+(\beta _{1}(t)-\beta _{0}(t))p_{1}(x;t)+\gamma _{0}(t)$
\end{proposition}

\begin{proof}
i) follows directly (\ref{w_o}) and initial conditions. ii) take square of
both sides of $p_{1}(x;t)=(x-\beta _{0}(t))/\alpha _{1}(t)$ and then
definition of $\hat{p}$. iii) multiply both sides of (\ref{w_o}) by $%
p_{n-1}(t)$ and integrate with respect to $\mu (dx,t)$. Secondly we note
that $\hat{p}_{0}(t)\allowbreak =\allowbreak 1$. v) This is so since $%
xp_{1}(x;t)\allowbreak =\allowbreak \alpha _{2}(t)p_{2}(x;t)+\beta
_{1}(t)p_{1}(x;t)+\gamma _{0}(t)\allowbreak =\allowbreak (\alpha
_{1}(t)p_{1}(t)+\beta _{0}(t))p_{1}(x;t)$ by i).
\end{proof}

\section{Harnesses\label{harness}}

The notion of harness as a special type of the stochastic process was
introduced by Hammersley in \cite{Ham67}. Recently it has been so to say
rediscovered and is intensively studied by Yor, Bryc, Weso\l owski, Matysiak
and others in \cite{Yor05}, \cite{BryMaWe}, \cite{brwe05}, \cite{BryMaWe07}, 
\cite{BryMaWe11}. In this chapter we are going to study conditions that are
to be satisfied for the OPM process to be a harness and also quadratic
harness.

We will use the extended slightly definitions of both these notions.

\begin{definition}
A Markov process $\mathbf{X=(}X_{t})_{t\in I}$ such that $\forall t\in
I:E\left\vert X_{t}\right\vert ^{r}\allowbreak <\allowbreak \infty ,$ $r\in 
\mathbb{N}$ is said to be $r-$harness if $\forall s<t<u:$ $E(X_{t}^{r}|%
\mathcal{F}_{s,u})$ is a polynomial of degree $r$ in $X_{s}$ and $X_{u}.$
\end{definition}

\begin{definition}
$1-$harness will be called simply harness while the process that is both $r-$%
harness for $r\allowbreak =\allowbreak 1,2$ will be called quadratic harness.
\end{definition}

\begin{remark}
Let us notice that every $r-$harness for $r\allowbreak =\allowbreak 1,\ldots
,N$ which is also $N-$TLI and $N-$MSC is also both $N-$rMPR and $N-$lMPR
however by no means conversely!.
\end{remark}

\begin{remark}
Notice also that for the OPM process is a harness iff 
\begin{equation}
E(p_{1}(X_{t};t)|\mathcal{F}_{s,u})=\hat{A}(s,t,u)p_{1}(X_{s};s)+\hat{B}%
(s,t,u)p_{1}(X_{u};u)+\hat{C}(s,t,u)  \label{har}
\end{equation}%
is satisfied for some functions $\hat{A},$ $\hat{B},$ $\hat{C}$ of $s,t,u\in
I$ and such that $s<t<u$. Similarly if a OPM process is a quadratic harness
if both (\ref{har}) is satisfied with some $\hat{A},$ $\hat{B},$ $\hat{C}$
and 
\begin{equation}
E(p_{2}(X_{t};t)|\mathcal{F}%
_{s,u})=Ap_{2}(X_{s};s)+Bp_{1}(X_{s};s)p_{1}(X_{u};u)+Cp_{2}(X_{u};u)+Dp_{1}(X_{s};s)+Ep_{1}(X_{u};u)+F.
\label{QH}
\end{equation}%
with some $A,$ $B,$ $C,$ $D,$ $E,$ $F$ of (depending on $s,t,u$ ) for all $%
s<t<u.$
\end{remark}

\subsection{Harnesses.}

First let us study harnesses. Hence we assume that (\ref{har}) is satisfied
with for some continuous functions of $s,$ $t,$ $u$. First of all notice
that $\hat{C}$ must be equal zero since $Ep_{1}(X_{t};t)\allowbreak
=\allowbreak 0$ and we get equality $0\allowbreak =\allowbreak 0\allowbreak
+\allowbreak 0\allowbreak +\allowbreak \hat{C}$ after calculating
expectation of both sides.

In the sequel we will denote for simplicity $\hat{p}_{1}(t)\allowbreak
=\allowbreak \hat{p}(t).$

We have the following simple lemma.

\begin{lemma}
\label{ha1}If an OPM process $\mathbf{X=(}X_{t})_{t\in I}$ is a harness
then: 
\begin{eqnarray}
\hat{A}(s,t,u) &=&\frac{\hat{p}(u)-\hat{p}(t)}{\hat{p}(u)-\hat{p}(s)},
\label{_A} \\
\hat{B}(s,t,u) &=&\frac{\hat{p}(t)-\hat{p}\left( s\right) }{\hat{p}\left(
u\right) -\hat{p}\left( s\right) }.  \label{_C}
\end{eqnarray}
\end{lemma}

\begin{proof}
Multiplying (\ref{har}) first by $p_{1}(X_{s};s)$ and then by $%
p_{1}(X_{u};u) $ and integrating we get equations:%
\begin{eqnarray*}
1 &=&\hat{A}+\hat{B}, \\
\hat{p}(t) &=&\hat{A}\hat{p}\left( s\right) +\hat{B}\hat{p}(u).
\end{eqnarray*}%
To get (\ref{_A}) and (\ref{_C}) is trivial.
\end{proof}

Now let us add assumption that our process has not only orthogonal
polynomial martingales but also it satisfies Condition \ref{sq_int}.
Consequently its transitional density is given by (\ref{trans}). We have the
following theorem.

\begin{theorem}
\label{lin_har}Let $\mathbf{X=(}X_{t})_{t\in I}$ be a Markov process with
orthogonal polynomial martingales $\left\{ p_{n}(x;t)\right\} _{n\geq -1}$
and measures $\mu (dx,t)$ and $\eta (dx,t;y,s)$ being respectively one
dimensional and transitional distributions satisfy Condition \ref{sq_int}.
Then $\mathbf{X}$ is a harness iff coefficients of 3-term recurrence (\ref%
{w_o}) are given by:%
\begin{eqnarray}
\beta _{n}(t) &=&\beta _{0}(t)+b_{n}\alpha _{1}(t)+\hat{b}_{n}\gamma _{0}(t),
\label{_be} \\
\alpha _{n}(t) &=&a_{n}\alpha _{1}(t)+\hat{a}_{n}\gamma _{0}(t),  \label{_al}
\\
\gamma _{n}(t) &=&c_{n}\alpha _{1}(t)+\hat{c}_{n}\gamma _{0}(t),
\label{_gam}
\end{eqnarray}%
where $\left\{ a_{n},\hat{a}_{n},b_{n},\hat{b}_{n},c_{n},\hat{c}_{n}\right\}
_{n\geq 0}$ are some number sequences such that $b_{0}=\hat{b}%
_{0}\allowbreak =\allowbreak \allowbreak \hat{a}_{0}\allowbreak \allowbreak
=\allowbreak a_{0}\allowbreak =\allowbreak \hat{a}_{1}\allowbreak
=\allowbreak \allowbreak c_{0}=\allowbreak 0,$ $a_{1}\allowbreak
=\allowbreak \hat{c}_{0}\allowbreak =\allowbreak 1$ and $\forall t\in I,$ $%
n\geq 0:$ $\alpha _{n}(t)\gamma _{n-1}(t)>0.$
\end{theorem}

\begin{proof}
Proof is shifted to Section \ref{dowody}.
\end{proof}

\begin{example}
If $\mathbf{X}$ is the $(\alpha ,q)-$OU process we have (basing on
Proposition \ref{poczatek}) $p_{n}(x;t)\allowbreak =\exp (\alpha
nt\allowbreak )H_{n}(x|q)$. Hence following (\ref{He}) we have $\alpha
_{n}(t)\allowbreak =\allowbreak \exp (-\alpha t),$ $\gamma
_{n}(t)\allowbreak =\allowbreak \lbrack n+1]_{q}\exp (\alpha t)$, $\hat{p}%
(t)\allowbreak =\allowbreak \gamma _{0}(t)/\alpha _{1}(t)\allowbreak
=\allowbreak \exp (2\alpha t)$. Thus equation (\ref{_al}) is satisfied with $%
a_{n}\allowbreak =\allowbreak 1,$ $\hat{a}_{n}\allowbreak =\allowbreak 0,$
equation (\ref{_gam}) with $c_{n}\allowbreak =\allowbreak 0,$ $\hat{c}%
_{n}=[n+1]_{q}$ while equation (\ref{_be}) with $b_{n}\allowbreak
=\allowbreak \hat{b}_{n}\allowbreak =\allowbreak 0.$

Now let us consider $q-$Wiener process. Recall that now $p_{n}(x;t)%
\allowbreak =\allowbreak t^{n/2}H_{n}(\frac{x}{\sqrt{t}}|q)$. Hence $\alpha
_{n}(t)\allowbreak =\allowbreak 1$ for $n\geq 1$ and $\gamma
_{0}(t)\allowbreak =\allowbreak \hat{p}(t)\allowbreak =\allowbreak \limfunc{%
var}(X_{t})\allowbreak =\allowbreak t$ as it follows from Proposition \ref%
{poczatek}. Hence $\gamma _{n}(t)\allowbreak =\allowbreak \lbrack
n+1]_{q}t\allowbreak =\allowbreak \lbrack n+1]_{q}\gamma _{0}(t)$ another
words $c_{n}\allowbreak =\allowbreak 0,$ $\hat{c}_{n}\allowbreak
=\allowbreak \lbrack n+1]_{q}.$
\end{example}

\subsection{Quadratic harnesses}

We will study now conditions leading to the fact that considered process $%
\mathbf{X}$ is a quadratic harness i.e. harness and such that $E(X_{t}^{2}|%
\mathcal{F}_{s,t})$ is a quadratic function of $X_{s}$ and $X_{u}.$ That is
assume that (\ref{QH}) is satisfied with some continuous functions $A,$ $B,$ 
$C,$ $D,$ $E,$ $F$ of $s,t,u$ for all $s<t<u,$ $s,t,u\in I.$

We have the following lemma:

\begin{lemma}
\label{quadrat}Let us denote $a,$ $\hat{a},$ $b,$ $\hat{b},$ $c,$ $\hat{c}$
six parameters defined by the conditions $\alpha _{2}(t)\allowbreak
=\allowbreak \alpha (t)(a+\hat{a}\hat{p}(t)),$ $\beta _{1}(t)-\beta
_{0}(t)\allowbreak =\allowbreak \alpha (t)(b+\hat{b}\hat{p}(t)),$ $\gamma
_{1}(t)\allowbreak =\allowbreak \alpha (t)(c+\hat{c}\hat{p}(t)),$ $(a+\hat{a}%
\hat{p}(t))(c+\hat{c}\hat{p}(t))>0$ for all $t\in I,$ where we denoted for
simplicity $\alpha (t)=\alpha _{1}(t)$ guaranteeing that the analyzed
process is a harness. Then 
\begin{equation*}
F\allowbreak =-B\hat{p}(s).
\end{equation*}%
Let us denote for simplicity:%
\begin{equation}
\kappa \allowbreak =\allowbreak (1+b\hat{b}+\hat{a}c),\lambda \allowbreak
=\allowbreak (a\hat{c}-\hat{a}c).  \label{kap}
\end{equation}

$\allowbreak $If $\kappa \allowbreak =\allowbreak \lambda \allowbreak
=\allowbreak 0$ then 
\begin{equation*}
A=\frac{\hat{p}(u)-\hat{p}(t)}{\hat{p}(u)-\hat{p}(s)}-B\frac{\hat{a}c}{\hat{c%
}},~C=\frac{\hat{p}(t)-\hat{p}(s)}{\hat{p}(u)-\hat{p}(s)}-\hat{a}B\hat{p}(s),
\end{equation*}%
and $B$ can be selected to be any nonzero function of $s<t<u$ and $D$ and $E$
are given by (\ref{de}), below.

If $\kappa ^{2}+\lambda ^{2}>0,$ then:%
\begin{gather}
B=\frac{a\hat{c}\lambda (\hat{p}(t)-\hat{p}(s))(\hat{p}(u)-\hat{p}(t))}{(a+%
\hat{a}\hat{p}(t))(\hat{p}(u)-\hat{p}(s))(\hat{p}(s)\hat{p}(u)\hat{a}\hat{c}%
\kappa +\hat{p}(u)a\hat{c}\kappa +\hat{p}(s)a\hat{c}(\kappa -\lambda
)+ac\kappa )},  \label{bb} \\
A=\frac{(a+\hat{a}\hat{p}(s))(\hat{p}(u)-\hat{p}(t))(\hat{p}(u)\hat{p}(t)%
\hat{a}\hat{c}\kappa +\hat{p}(u)a\hat{c}\kappa +\hat{p}(t)a\hat{c}(\kappa
-\lambda )+ac\kappa )}{(a+\hat{a}\hat{p}(t))(\hat{p}(u)-\hat{p}(s))(\hat{p}%
(s)\hat{p}(u)\hat{a}\hat{c}\kappa +\hat{p}(u)a\hat{c}\kappa +\hat{p}(s)a\hat{%
c}(\kappa -\lambda )+ac\kappa )},  \label{aa} \\
C=\frac{(a+\hat{a}\hat{p}(u))(\hat{p}(t)-\hat{p}(s))(\hat{p}(t)\hat{p}(s)%
\hat{a}\hat{c}\kappa +\hat{p}(s)a\hat{c}(\kappa +\lambda )+\hat{p}(t)a\hat{c}%
\kappa +ac\kappa )}{(a+\hat{a}\hat{p}(t))(\hat{p}(u)-\hat{p}(s))(\hat{p}(s)%
\hat{p}(u)\hat{a}\hat{c}\kappa +\hat{p}(u)a\hat{c}\kappa +\hat{p}(s)a\hat{c}%
(\kappa -\lambda )+ac\kappa )},  \label{cc}
\end{gather}

and$~$

\begin{equation}
D=-bB,~E=-\hat{b}B\hat{p}(s).  \label{de}
\end{equation}
\end{lemma}

\begin{proof}
Proof is shifted to Section \ref{dowody}.
\end{proof}

\begin{remark}
Notice that following Proposition \ref{poczatek} we have $\hat{p}%
_{2}(t)\allowbreak =\allowbreak \hat{p}(t)\frac{a+\hat{a}\hat{p}(t)}{c+\hat{c%
}\hat{p}(t)}$. Also as a consequence of Favard's Theorem we must have $(a+%
\hat{a}\hat{p}(t))(c+\hat{c}\hat{p}(t))\allowbreak >\allowbreak 0$. Further
keeping in mind interpretation of $\hat{p}_{2}(t)$ as the expectation of the
'square bracket' of martingale $p_{2}(X_{t};t)$ we deduce that $\hat{p}%
_{2}(t)$ must be non-decreasing. If interval $I$ is bounded then the only
restrictions on parameters $a,$ $\hat{a},$ $c,$ $\hat{c}$ are 
\begin{eqnarray*}
(a+\hat{a}\hat{p}(t))(c+\hat{c}\hat{p}(t))\allowbreak &>&\allowbreak 0, \\
\hat{p}(s)\frac{a+\hat{a}\hat{p}(s)}{c+\hat{c}\hat{p}(s)} &<&\hat{p}(t)\frac{%
a+\hat{a}\hat{p}(t)}{c+\hat{c}\hat{p}(t)},
\end{eqnarray*}%
for all $s,t\in I$ and $s<t.$

Now suppose that $I$ is unbounded. If $I$ is unbounded from the right and
suppose that $\hat{p}(t)$ is bounded. This would mean that martingale $%
p_{1}(X_{t};t)$ converges to finite limit a.s. and in $L_{2}$. Consequently $%
X_{t}/\alpha _{1}(t)$ would have a finite a.s. limit. This is rather
uninteresting case from the point of view of stochastic processes theory.
Hence let us assume from now on that $\hat{p}(t)$ is unbounded if $I$ is
unbounded from the right hand side. Similarly let us assume that $\hat{p}%
(t)\longrightarrow 0$ as $t\longrightarrow -\infty $ that is if $I$ is
unbounded from the left hand side or if left boundary of $I$ is $0,$ i.e. if 
$I\allowbreak =\allowbreak \lbrack 0,r],$ then assume that $\hat{p}%
(0)\allowbreak =\allowbreak 0$. We have also $\frac{d}{dt}(\hat{p}(t)\frac{a+%
\hat{a}\hat{p}(t)}{c+\hat{c}\hat{p}(t)})\allowbreak =\allowbreak \frac{\hat{a%
}\hat{c}\hat{p}(t)^{2}+2a\hat{c}\hat{p}(t)+ac}{(c+\hat{c}\hat{p}(t))^{2}}%
\hat{p}^{\prime }(t)$ and $\hat{p}^{\prime }(t)>0$. Consequently if $I$ is
unbounded from the left hand side or $I\allowbreak =\allowbreak \lbrack 0,r]$
we must have $ac\allowbreak \geq \allowbreak 0$ and if $I$ is unbounded from
the right hand side we must have $\hat{a}\hat{c}\allowbreak \geq \allowbreak
0.$ If $\hat{a}\hat{c}\neq 0$ then if $I$ is unbounded from the right we
must have $\hat{a}\hat{c}>0$.
\end{remark}

We have also the following corollary concerning particular cases:

\begin{corollary}
\label{part_case} i) If $\kappa \allowbreak =\allowbreak 0$ and $\lambda
\neq 0$ then we have:%
\begin{eqnarray*}
B &=&-\frac{(\hat{p}(t)-\hat{p}(s))(\hat{p}(u)-\hat{p}(t))}{(a+\hat{a}\hat{p}%
(t))(\hat{p}(u)-\hat{p}(s))\hat{p}(s)}, \\
A &=&\frac{(a+\hat{a}\hat{p}(s))(\hat{p}(u)-\hat{p}(t))\hat{p}(t)}{(a+\hat{a}%
\hat{p}(t))(\hat{p}(u)-\hat{p}(s))\hat{p}(s)}, \\
C &=&\frac{(a+\hat{a}\hat{p}(u))(\hat{p}(t)-\hat{p}(s))}{(a+\hat{a}\hat{p}%
(t))(\hat{p}(u)-\hat{p}(s))}.
\end{eqnarray*}

ii)If $\kappa \neq \allowbreak 0$ and $\lambda \allowbreak =0$ then 
\begin{eqnarray*}
B &=&0, \\
A &=&\frac{(a+\hat{a}\hat{p}(s))(\hat{p}(u)-\hat{p}(t))(\hat{p}(u)\hat{p}(t)%
\hat{a}\hat{c}+\hat{p}(u)a\hat{c}+\hat{p}(t)a\hat{c}+ac)}{(a+\hat{a}\hat{p}%
(t))(\hat{p}(u)-\hat{p}(s))(\hat{p}(s)\hat{p}(u)\hat{a}\hat{c}+\hat{p}(u)a%
\hat{c}+\hat{p}(s)a\hat{c}+ac)}=\frac{(\hat{p}(u)-\hat{p}(t))}{(\hat{p}(u)-%
\hat{p}(s))}, \\
C &=&\frac{(a+\hat{a}\hat{p}(u))(\hat{p}(t)-\hat{p}(s))(\hat{p}(t)\hat{p}(s)%
\hat{a}\hat{c}+\hat{p}(s)a\hat{c}+\hat{p}(t)a\hat{c}+ac)}{(a+\hat{a}\hat{p}%
(t))(\hat{p}(u)-\hat{p}(s))(\hat{p}(s)\hat{p}(u)\hat{a}\hat{c}+\hat{p}(u)a%
\hat{c}+\hat{p}(s)a\hat{c}+ac)}=\frac{(\hat{p}(t)-\hat{p}(s))}{(\hat{p}(u)-%
\hat{p}(s))}.
\end{eqnarray*}

iii) Otherwise if $\kappa \allowbreak \neq \allowbreak 0$ and $\lambda
\allowbreak \allowbreak \neq \allowbreak 0$ then after dividing both
denominators and denumerators by $\kappa $ we get: 
\begin{eqnarray*}
B &=&\frac{a\hat{c}\frac{\lambda }{\kappa }(\hat{p}(t)-\hat{p}(s))(\hat{p}%
(u)-\hat{p}(t))}{(a+\hat{a}\hat{p}(t))(\hat{p}(u)-\hat{p}(s))(\hat{p}(s)\hat{%
p}(u)\hat{a}\hat{c}+\hat{p}(u)a\hat{c}+\hat{p}(s)a\hat{c}\frac{\kappa
-\lambda }{\kappa }+ac)}, \\
A &=&\frac{(a+\hat{a}\hat{p}(s))(\hat{p}(u)-\hat{p}(t))(\hat{p}(u)\hat{p}(t)%
\hat{a}\hat{c}+\hat{p}(u)a\hat{c}+\hat{p}(t)a\hat{c}\frac{\kappa -\lambda }{%
\kappa }+ac)}{(a+\hat{a}\hat{p}(t))(\hat{p}(u)-\hat{p}(s))(\hat{p}(s)\hat{p}%
(u)\hat{a}\hat{c}+\hat{p}(u)a\hat{c}+\hat{p}(s)a\hat{c}\frac{\kappa -\lambda 
}{\kappa }+ac)}, \\
C &=&\frac{(a+\hat{a}\hat{p}(u))(\hat{p}(t)-\hat{p}(s))(\hat{p}(t)\hat{p}(s)%
\hat{a}\hat{c}+\hat{p}(s)a\hat{c}\frac{\kappa -\lambda }{\kappa }+\hat{p}(t)a%
\hat{c}+ac)}{(a+\hat{a}\hat{p}(t))(\hat{p}(u)-\hat{p}(s))(\hat{p}(s)\hat{p}%
(u)\hat{a}\hat{c}+\hat{p}(u)a\hat{c}+\hat{p}(s)a\hat{c}\frac{\kappa -\lambda 
}{\kappa }+ac)}.
\end{eqnarray*}%
If $a\hat{c}\neq 0$ then one can simplify these expressions further by
dividing both denominators and denumerators by $a\hat{c}$ and we get: 
\begin{eqnarray}
B &=&\frac{\frac{\lambda }{\kappa }(\hat{p}(t)-\hat{p}(s))(\hat{p}(u)-\hat{p}%
(t))}{(a+\hat{a}\hat{p}(t))(\hat{p}(u)-\hat{p}(s))(\hat{p}(s)\hat{p}(u)\frac{%
\hat{a}}{a}+\hat{p}(u)+\hat{p}(s)\frac{\kappa -\lambda }{\kappa }+\frac{c}{%
\hat{c}})},  \label{bmw1} \\
A &=&\frac{(a+\hat{a}\hat{p}(s))(\hat{p}(u)-\hat{p}(t))(\hat{p}(u)\hat{p}(t)%
\frac{\hat{a}}{a}+\hat{p}(u)+\hat{p}(t)\frac{\kappa -\lambda }{\kappa }+%
\frac{c}{\hat{c}})}{(a+\hat{a}\hat{p}(t))(\hat{p}(u)-\hat{p}(s))(\hat{p}(s)%
\hat{p}(u)\frac{\hat{a}}{a}+\hat{p}(u)+\hat{p}(s)\frac{\kappa -\lambda }{%
\kappa })+\frac{c}{\hat{c}})},  \label{bmw2} \\
C &=&\frac{(a+\hat{a}\hat{p}(u))(\hat{p}(t)-\hat{p}(s))(\hat{p}(t)\hat{p}(s)%
\frac{\hat{a}}{a}+\hat{p}(s)\frac{\kappa -\lambda }{\kappa }+\hat{p}(t)+%
\frac{c}{\hat{c}})}{(a+\hat{a}\hat{p}(t))(\hat{p}(u)-\hat{p}(s))(\hat{p}(s)%
\hat{p}(u)\frac{\hat{a}}{a}+\hat{p}(u)+\hat{p}(s)\frac{\kappa -\lambda }{%
\kappa }+\frac{c}{\hat{c}})}.  \label{bmw3}
\end{eqnarray}
\end{corollary}

\begin{remark}
\label{ogran}Assume that $a\hat{c}\neq 0$ and that $\hat{p}%
(u)\longrightarrow \infty $ if $u\longrightarrow \infty $ and $\hat{p}\left(
u\right) \longrightarrow 0$ if $u\longrightarrow l$ where $l$ is defined by $%
I\allowbreak =\allowbreak \lbrack l,\infty ).$ Then following argument use d
by Bryc et. al. in \cite{BryMaWe07}, p. 5456 we can deduce that $\frac{%
\kappa -\lambda }{\kappa }\allowbreak \leq \allowbreak 1+2\sqrt{\hat{a}c/(a%
\hat{c})}.$
\end{remark}

\begin{corollary}
\label{bryc}In \cite{BryMaWe07} Bryc et al. considered Markov processes $%
\mathbf{X\allowbreak =\allowbreak (}X_{t})_{t\geq 0}$ that start from zero
at zero, are defined on $[0,\infty )$. Moreover these processes are assumed
to satisfy the following normalizing conditions: 1) $EX_{t}=0,$ 2) $\limfunc{%
var}(X_{t})\allowbreak =\allowbreak t,$ $\limfunc{cov}(X_{t},X_{s})%
\allowbreak =\allowbreak \min (s,t),$ 3) $E(X_{t}|\mathcal{F}_{\leq
s})\allowbreak =\allowbreak X_{s}$ and $E(X_{t}^{2}|\mathcal{F}_{\leq
s})\allowbreak =\allowbreak X_{s}^{2}+t-s$. for all $s\leq t$. From these
assumptions we deduce that $p_{1}(x;t)\allowbreak =\allowbreak x,$ $\hat{p}%
(t)\allowbreak =\allowbreak \hat{p}_{1}(t)\allowbreak =\allowbreak t,$ $%
\alpha _{1}(t)\allowbreak =\allowbreak 1$. Moreover since $\hat{p}(t)$
ranges from $0$ to infinity and we have conditions : $0\leq \alpha
_{2}(t)\gamma _{1}(t)\allowbreak =\allowbreak (a+\hat{a}t)(c+\hat{c}t);$ $%
0<\alpha _{2}(t)\allowbreak =\allowbreak a+\hat{a}t$ we deduce that $a>0,$ $%
ac\geq 0,$ $\hat{a}\hat{c}\geq 0$ and $ac\allowbreak \allowbreak
+\allowbreak \hat{a}\hat{c}>0.\allowbreak $ Besides we have (basing on our
assumptions) $\gamma _{1}(t)\hat{p}(t)\allowbreak =\allowbreak \alpha _{2}(t)%
\hat{p}_{2}(t)$ which leads to the relationship: $\allowbreak (c+\hat{c}%
t)t=\allowbreak \hat{p}_{2}(t)(a+\hat{a}t)$. If $\hat{c}\allowbreak
=\allowbreak 0$ then $\hat{p}_{2}(t)$ could not be increasing. Hence we
deduce that $\hat{c}>0$. Thus examining equation (\ref{bmw1})-(\ref{bmw3})
we deduce that there are $3$ independent parameters $\hat{a}/a$ which Bryc
at al. in \cite{BryMaWe07} called $\sigma ,$ $c/\hat{c}$ which was called $%
\tau $ in and $\frac{\kappa -\lambda }{\kappa }$ which was denoted $-q$ in 
\cite{BryMaWe07}. Notice that $\frac{\lambda \allowbreak }{\kappa }%
=\allowbreak 1+q$ in this notation. Following Proposition \ref{poczatek},iv)
we deduce that $EX_{t}^{3}\allowbreak =\allowbreak \beta _{1}(t)t\allowbreak
=\allowbreak (b+\hat{b}t)t$. On the other hand considering \cite{BryMaWe07}%
,(4.13) for $x\allowbreak =\allowbreak X_{t},$ then multiplying both sides
of so obtained expression $by$ $X_{t}$, taking expectation of both sides and
on the way making use of orthogonality $p_{2}$ and $p_{1}$ we can calculate $%
EX_{t}^{3}$ in the setting of \cite{BryMaWe07} obtaining 
\begin{equation*}
EX_{t}^{3}\allowbreak =\allowbreak \frac{(\eta +\sigma \theta )t^{2}+(\eta
\tau +\theta )t}{1-\sigma \tau }.
\end{equation*}%
Comparing these two expression for $EX_{t}^{3}$ we get interpretation of our
parameters $b$ and $\hat{b}$ in terms of $\sigma ,$ $\tau $, $q$ and other
two parameters used by Bryc et. al. i.e. $\eta $ and $\theta .$ Besides
since $\hat{p}_{2}(t)\allowbreak =\allowbreak \frac{t(c+\hat{c}t)}{a+\hat{a}t%
}\allowbreak =\allowbreak \frac{t(\tau +t)}{1+\sigma t}$ is to be increasing
we get (by calculating derivative) that for all $t\geq 0:$ $\sigma
t^{2}+2t+\tau \geq 0$. Consequently that $\sigma ,\tau \geq 0$ and $%
q\allowbreak \leq \allowbreak 1+2\sqrt{\sigma \tau }$ by the Remark \ref%
{ogran}.
\end{corollary}

\begin{theorem}
\label{main}Let $\mathbf{X=(}X_{t})_{t\in I}$ be a Markov process with
orthogonal polynomial martingales $\left\{ p_{n}(x;t)\right\} _{n\geq -1}$
and $\mu (dx,t)$ and $\eta (dx,t;y,s)$ as one dimensional and transitional
distributions satisfying Condition \ref{sq_int}. Then $\mathbf{X}$ is a
quadratic harness iff 3-term recurrence satisfied by polynomials $p_{n}$ is
given by (\ref{w_o}) with coefficients $\alpha _{n}(t),$ $\beta _{n}(t)$, $%
\gamma _{n}(t)$ defined by (\ref{_al}), (\ref{_be}) and (\ref{_gam}) with
the system of $6$ number sequences $\left\{ a_{n},\hat{a}_{n},b_{n},\hat{b}%
_{n},c_{n},\hat{c}_{n}\right\} $ that satisfy the following system of $5$
recursive equations:%
\begin{gather}
\kappa (\hat{a}\hat{c}a_{n+1}a_{n+2}\allowbreak +\allowbreak \hat{a}_{n+1}%
\hat{a}_{n+2}ac-a\hat{c}a_{n+1}\hat{a}_{n+2})\allowbreak =(\kappa -\lambda )a%
\hat{c}\hat{a}_{n+1}a_{n+2},  \label{e1} \\
\kappa (\hat{a}\hat{c}c_{n-2}c_{n-1}-a\hat{c}\hat{c}_{n-2}c_{n-1}+ac\hat{c}%
_{n-2}\hat{c}_{n-1})=(\kappa -\lambda )a\hat{c}c_{n-2}\hat{c}_{n-1},
\label{e2} \\
\kappa (ac\hat{a}_{n+1}(\hat{b}_{n+1}+\hat{b}_{n})+\hat{a}\hat{c}%
a_{n+1}(b_{n+1}+b_{n})-(\hat{b}c-b\hat{c})a\hat{a}_{n+1}-(\hat{a}b-a\hat{b})%
\hat{c}a_{n+1})  \label{e3} \\
=a\hat{c}((\kappa -\lambda )(\hat{a}_{n+1}b_{n+1}\allowbreak +\allowbreak
a_{n+1}\hat{b}_{n})+\kappa (a_{n+1}\hat{b}_{n+1}+\hat{a}_{n+1}b_{n})),
\label{e3_1} \\
\kappa (\hat{a}\hat{c}c_{n-1}(b_{n}+b_{n-1})+\allowbreak ac\hat{c}_{n-1}(%
\hat{b}_{n}+\hat{b}_{n-1})-\allowbreak \allowbreak (\hat{b}c-b\hat{c})a\hat{c%
}_{n-1}-(\hat{a}b-a\hat{b})\hat{c}c_{n-1})  \label{e4} \\
=a\hat{c}((\kappa -\lambda )(c_{n-1}\hat{b}_{n}+\hat{c}_{n-1}b_{n-1})+\kappa
(\hat{c}_{n-1}b_{n}+c_{n-1}\hat{b}_{n-1})),  \label{e4_1} \\
\kappa (ac(\hat{a}_{n+1}\hat{c}_{n}+\hat{a}_{n}\hat{c}_{n-1}+\hat{b}_{n}(%
\hat{b}_{n}-\hat{b}))\allowbreak +\allowbreak \hat{a}\hat{c}%
(a_{n+1}c_{n}+a_{n}c_{n-1}+b_{n}(b_{n}-b))+a\hat{c})=  \label{e5} \\
a\hat{c}((\kappa -\lambda )(\hat{a}_{n+1}c_{n}+a_{n}\hat{c}_{n-1}+b_{n}\hat{b%
}_{n})+\kappa (a_{n+1}\hat{c}_{n}+\hat{a}_{n}c_{n-1}-b\hat{b}+(b_{n}-b)(\hat{%
b}_{n}-\hat{b})),  \label{e5_1}
\end{gather}%
with initial conditions : $a_{1}\allowbreak =\allowbreak a,$ $\hat{a}%
_{1}\allowbreak =\allowbreak \hat{a},$ $b_{1}\allowbreak =\allowbreak b,$ $%
\hat{b}_{1}\allowbreak =\allowbreak \hat{b},$ $c_{1}\allowbreak =\allowbreak
c$ and $\hat{c}_{1}\allowbreak =\allowbreak \hat{c},$ and such that $\forall
t\in I,$ $n\geq 0:$ $(a_{n}+\hat{a}_{n}\hat{p}(t))(c_{n}+\hat{c}_{n}\hat{p}%
(t))>0.$
\end{theorem}

\begin{proof}
Long tedious proof is shifted to Section \ref{dowody}.
\end{proof}

\begin{example}
Let us consider three examples. The $q-$Wiener and $(\alpha ,q)-$%
Ornstein--Uhlenbeck processes described above and the Poisson process.

Let us start with $q-$Wiener process. Following what was presented in
Subsection \ref{example} we have $p_{n}(x;t)\allowbreak =\allowbreak
t^{n/2}H_{n}(x/\sqrt{t})$. Hence following (\ref{He}) we have:%
\begin{equation*}
xp_{n}(x;t)=p_{n+1}(x;t)+t[n]_{q}p_{n-1}(x;t).
\end{equation*}%
Besides $\hat{p}(t)\allowbreak =\allowbreak \hat{p}_{1}(t)\allowbreak
=\allowbreak Ep_{1}^{2}(X_{t};t)\allowbreak =\allowbreak t$. So $\alpha
_{n+1}(t)\allowbreak =\allowbreak 1+0t$ and thus $a_{n}\allowbreak
=\allowbreak 1,$ $\hat{a}_{n}\allowbreak =\allowbreak 0$. Further $\beta
_{n}(t)\allowbreak =\allowbreak 0$ so $b_{n}\allowbreak =\allowbreak \hat{b}%
_{n}\allowbreak =\allowbreak 0$. Finally $\gamma _{n-1}(t)\allowbreak
=\allowbreak t[n]_{q},$ so $c_{n}\allowbreak =\allowbreak 0,$ $\hat{c}%
_{n}\allowbreak =\allowbreak \lbrack n+1]_{q}$. Consequently values of
parameters are $a\allowbreak =\allowbreak a_{1},$ $\hat{a}\allowbreak
=\allowbreak \hat{a}_{1}\allowbreak =\allowbreak 0,$ $b=\hat{b}\allowbreak
=\allowbreak 0,$ $c=c_{1}\allowbreak =\allowbreak 0,$ $\hat{c}\allowbreak
=\allowbreak c_{1}\allowbreak =\allowbreak 1+q,$ $\kappa \allowbreak
=\allowbreak 1+b\hat{b}+\hat{a}c\allowbreak =\allowbreak 1,$ $\lambda
\allowbreak =\allowbreak a\hat{c}-\hat{a}c\allowbreak =\allowbreak 1+q,$ $%
\kappa -\lambda \allowbreak =\allowbreak -q$. Functions $A,$ $B,$ $C,$ $D,$ $%
E,$ $F$ are now%
\begin{eqnarray*}
A\allowbreak &=&\allowbreak \frac{(u-t)(u-qt)}{(u-s)(u-qs)},~B=\frac{%
(1+q)(t-s)(u-t)}{(u-s)(u-qs)}, \\
C &=&\frac{(t-s)(t-qs)}{(u-s)(u-qs)},~D\allowbreak =\allowbreak 0,~E=0,F=-sB.
\end{eqnarray*}%
One can also check that equations (\ref{e1}), (\ref{e2}), (\ref{e3}), (\ref%
{e3_1}), (\ref{e4}), (\ref{e4_1}) are trivially satisfied since $b=\hat{b}%
\allowbreak =\allowbreak 0$ either $\hat{a}=\allowbreak 0$ or $c\allowbreak
=\allowbreak 0$ and $\hat{a}_{n}\allowbreak =\allowbreak 0,$ $%
c_{n}\allowbreak =\allowbreak 0$. To check the equation (\ref{e5}) (\ref%
{e5_1}) we have:%
\begin{equation*}
(1+q)\allowbreak =\allowbreak (1+q)(-q[n]_{q}+[n+1]_{q}),
\end{equation*}%
which is true. Hence we deduce that $q-$Wiener process is a quadratic
harness.

Now let us analyze $(\alpha ,q)-$OW process. Following what was presented in
Subsection \ref{example} we have $p_{n}(x;t)\allowbreak =\allowbreak \exp
(n\alpha t)H_{n}(x|q)$. Hence following (\ref{He}) we have:%
\begin{equation*}
xp_{n}(x;t)=e^{-\alpha t}p_{n+1}(x;t)+e^{\alpha t}[n]_{q}p_{n-1}(x;t).
\end{equation*}%
In particular $\alpha _{1}(t)\allowbreak =\allowbreak e^{-\alpha t},$ $%
\gamma _{0}(t)\allowbreak =\allowbreak e^{\alpha t}$. Besides $\hat{p}\left(
t\right) \allowbreak =\allowbreak \hat{p}_{1}(t)\allowbreak =\allowbreak
Ep_{1}^{2}(X_{t};t)\allowbreak =\allowbreak e^{2\alpha t}\allowbreak
=\allowbreak \gamma _{0}(t)/\alpha _{1}(t)$. Following equations (\ref{_be}%
), (\ref{_al}) and (\ref{_gam}) we have $\alpha _{n+1}(t)\allowbreak
=\allowbreak \allowbreak e^{-\alpha t}\allowbreak =\allowbreak 1\alpha
_{1}(t)+0\gamma _{0}(t)\allowbreak $ so $a_{n}\allowbreak =\allowbreak 1$
and $\hat{a}_{n}\allowbreak =\allowbreak 0,$ $\gamma _{n-1}(t)\allowbreak
=\allowbreak e^{\alpha t}[n]_{q}\allowbreak =\allowbreak 0\alpha
_{1}(t)\allowbreak +\allowbreak \lbrack n]_{q}\allowbreak \gamma _{0}(t),$
so $c_{n}\allowbreak =\allowbreak 0$ and $\hat{c}_{n}\allowbreak
=\allowbreak \lbrack n+1]_{q}$. Of course like before $b\allowbreak
=\allowbreak \hat{b}\allowbreak =\allowbreak b_{n}\allowbreak =\allowbreak 
\hat{b}_{n}\allowbreak =\allowbreak 0$. Functions $A,$ $B,$ $C,$ $D,$ $E,$ $%
F $ are now%
\begin{eqnarray*}
A\allowbreak &=&\allowbreak \frac{(e^{2\alpha u}-e^{2\alpha t})(e^{2\alpha
u}-qe^{2\alpha t})}{(e^{2\alpha u}-e^{2\alpha s})(e^{2\alpha u}-qe^{2\alpha
s})},~B=\frac{(1+q)(e^{2\alpha t}-e^{2\alpha s})(e^{2\alpha u}-e^{2\alpha t})%
}{(e^{2\alpha u}-e^{2\alpha s})(e^{2\alpha u}-qe^{2\alpha s})}, \\
C &=&\frac{(e^{2\alpha t}-e^{2\alpha s})(e^{2\alpha t}-qe^{2\alpha s})}{%
(e^{2\alpha u}-e^{2\alpha s})(e^{2\alpha u}-qe^{2\alpha s})},~D\allowbreak
=\allowbreak 0,~E=0,F=-e^{2\alpha s}B.
\end{eqnarray*}%
Like in the case of $q-$Wiener process equations (\ref{e1}), (\ref{e2}), (%
\ref{e3}), (\ref{e3_1}), (\ref{e4}), (\ref{e4_1}) are trivially satisfied
since $b=\hat{b}\allowbreak =\allowbreak 0$ either $\hat{a}=\allowbreak 0$
or $c\allowbreak =\allowbreak 0$ and $\hat{a}_{n}\allowbreak =\allowbreak 0,$
$c_{n}\allowbreak =\allowbreak 0$. Equation (\ref{e5}) (\ref{e5_1}) has the
same form as in the $q-$Wiener process. Hence we deduce that $(\alpha ,q)-$%
OW process is a quadratic harness.

Let us now turn our attention to the Poisson process with parameter $\mu $.
Following \cite{Koek} (section 1.12) we deduce that $p_{n}(x;t)\allowbreak
=\allowbreak (-\mu t)^{n}C_{n}(x;\mu t),$ where $C_{n}$ denotes Charlier
polynomial as defined by (\cite{Koek}, 1.12.1). Moreover following \cite%
{Koek}, (1.12.4) we have%
\begin{equation*}
xp_{n}(x;t)=p_{n+1}(x;t)+(n+\mu t)p_{n}(x;t)+n\mu tp_{n-1}(x;t).
\end{equation*}%
Besides following \cite{Koek}(1.12.2) we deduce that $\hat{p}\left( t\right)
\allowbreak =\allowbreak \hat{p}_{1}(t)\allowbreak =\lambda t$. Thus
parameters are the following: $a\allowbreak =\allowbreak 1,$ $\hat{a}%
\allowbreak =\allowbreak 0,$ $b\allowbreak =\allowbreak 1,$ $\hat{b}%
\allowbreak =\allowbreak 0,$ $c\allowbreak =\allowbreak 0,$ $\hat{c}%
\allowbreak =\allowbreak 2,$ $\kappa =1,$ $\lambda \allowbreak =\allowbreak
2 $ since parameters $b$ and $\hat{b}$ are defined by $\beta
_{1}(t)\allowbreak -\allowbreak \beta _{0}(t)\allowbreak =\allowbreak
b\allowbreak +\allowbreak \hat{b}\hat{p}\left( t\right) $. Following
equations (\ref{_be}), (\ref{_al}) and (\ref{_gam}) we have $\alpha
_{n+1}(t)\allowbreak =\allowbreak 1+0\mu t,$ $\beta _{n}(t)\allowbreak
=\allowbreak n,$ $\gamma _{n-1}(t)\allowbreak =\allowbreak n\mu t$ and thus
consequently $a_{n}\allowbreak =\allowbreak 1,$ $\hat{a}_{n}\allowbreak
=\allowbreak 0,$ $b_{n}\allowbreak =\allowbreak n,$ $\hat{b}_{n}\allowbreak
=\allowbreak 0,$ $c_{n}\allowbreak =\allowbreak 0,$ $\hat{c}_{n}\allowbreak
=\allowbreak n+1$. Functions $A,$ $B,$ $C,$ $D,$ $E,$ $F$ are now%
\begin{eqnarray*}
A\allowbreak &=&\allowbreak \frac{(u-t)^{2}}{(u-s)^{2}},~B=\frac{2(u-t)(t-s)%
}{(u-s)^{2}}, \\
C &=&\frac{(t-s)^{2}}{(u-s)^{2}},~D\allowbreak =\allowbreak -B,~E=-sB,F=-sB.
\end{eqnarray*}%
Thus we deduce that Poisson process is a harness. It is not new result since
this fact was already known to Jacod et.al. as shown in \cite{Jaco88}. As
far as the property of being a quadratic harness is concerned we deduce that
equations (\ref{e1}), (\ref{e2}), (\ref{e3}), (\ref{e3_1}) are trivially
satisfied since either $\hat{a}=\allowbreak 0$ or $c\allowbreak =\allowbreak
0,$ $\hat{a}_{n}\allowbreak =\allowbreak 0,$ $c_{n}\allowbreak =\allowbreak
0 $. As far as equations (\ref{e4}), (\ref{e4_1}) is concerned we on the
left hand sides we have $2n(n-n+1)$ while on the right hand side we get $2n$%
. Finally equation (\ref{e5}) (\ref{e5_1}) leads to the following identity: $%
2\allowbreak =\allowbreak 2(-n+n+1)$. hence we deduce that process $\left\{
Z_{t}\right\} _{t\geq 0}$ is a quadratic harness. This confirms result of 
\cite{BryMaWe07}.
\end{example}

\begin{remark}
Let us remark that in recent years many more examples of quadratic harnesses
were found by Bryc et. al. To mention only \cite{BRWE12}, \cite{BryBo}, \cite%
{BryMaWe11} or \cite{SzabHar}.
\end{remark}

\section{Open problems and remarks\label{open}}

Quadratic harnesses have been defined for square integrable processes.
However majority of examples (e.g. in \cite{BryMaWe},\cite{BryMaWe07}, \cite%
{BryMaWe11}, \cite{BryWe}, \cite{BryWe10}) deal with quadratic harnesses
that have all moments and belong to the class OPM. Recently in \cite{BRWE12}
the authors showed that there exist quadratic harnesses that do not have all
moments.

There is an immediate interesting question are there $2-$harnesses that are
not harnesses? Or more generally are there $r+1-$harnesses that are not $r-$%
harnesses?

At first sight it seems that yes but I do not know any examples.

Recalling Theorem \ref{lin_har} we notice that condition of being a harness
defines $6$ number sequences $\{a_{n},\hat{a}_{n},b_{n},\hat{b}_{n},c_{n},%
\hat{c}_{n}\}_{n\geq 1}$. Further little reflection shows that being a
harness and $r-$harness would require more and more equations to be
satisfied by the same number of $6$ parameters sequences. When $r$ was equal 
$2$ we had $5$ conditions, and generally for any $r$ we would have $2r+1$
conditions to be satisfied by $6$ number sequences. Hence to be a $r-$%
harness would be more and more difficult. Are there processes that are
harnesses for every $r\in \mathbb{N}$ (let's call this property as being a
total harness)? The answer is yes. It turns out that not only classical
Wiener and Ornstein--Uhlenbeck process are total harnesses but the result
presented in \cite{SzablAW}, Theorem 2. can be interpreted in such a way
that also $q-$Wiener and $(\alpha ,q)-$OU processes are total harnesses.

Are there other total harnesses? Considering as $\left\{ X_{t}\right\}
_{t\geq 0}$ the Poisson process with parameter say equal to $\mu ,$ little
reflection leads to the conclusion that the conditional distribution of $%
X_{t}|X_{s}=k,X_{u}=n$ for $s<t<u$ is equal to the distribution of $k+Y$
where $Y\allowbreak \sim \allowbreak Bin(n-k,\frac{t-s}{u-s}),$ where $%
Bin(n,p)$ denotes binomial distribution with parameters $n\in \mathbb{N}$
and $p\in \lbrack 0,1]$. Following well known properties of binomial
distributions we see that $n-th$ moment of $k+Y$ is a polynomial of degree
at most $n$ of $k$ and $n$. Hence Poisson process is also a total harness.
But are they the only ones? What are the necessary and sufficient conditions
to be satisfied by $\{a_{n},\hat{a}_{n},b_{n},\hat{b}_{n},c_{n},\hat{c}%
_{n}\}_{n\geq 1}$ for being a total harness?

There are also questions connected with the Theorem \ref{expansion}. If we
know that transitional and marginal distributions are identifiable by
moments, satisfy Condition \ref{sq_int} and there exist family of
polynomials satisfying (\ref{Ort}) and (\ref{Mart}) then these distributions
must me interrelated by the formula (\ref{trans}). However not necessarily
conversely. In general expressions $\sum_{n\geq 0}\frac{1}{\hat{p}_{n}(t)}%
p_{n}(x,t)p_{n}(y,s)$ are not nonnegative on the support on the measure that
makes polynomials $\left\{ p_{n}\right\} $ orthogonal. For what families of
orthogonal polynomials $\left\{ p_{n}(x,t)\right\} $ with respect to $\mu
(.,t)$ this expressions are nonnegative on the support of $\mu ?$ Examples
presented above show that there exist nontrivial families of polynomials
having this property.

There is another interesting more general and more difficult question that
we do not know the answer for. Namely as it follows from Theorem \ref%
{expansion} there are at least as many OPM processes as there are marginal
measures $\mu (.,t),$ $t\in I$. Another words having family $\mu (.,t),$ $%
t\in I$ of probability measures we have family of orthogonal polynomials $%
p_{n}(x;t)$ and by Theorem \ref{expansion} we have transitional measure. In
general within MPR class we have for all $t\in I$ and $n\geq 1:$ 
\begin{equation*}
Ep_{n}(X_{t};t)|\mathcal{F}_{\leq s})=q_{n}(X_{s};s,t),
\end{equation*}%
where $q_{n}$ denotes polynomial of degree not exceeding $n$ and $p_{n}$ as
before orthogonal polynomial of marginal distribution $\mu $. In general $%
q_{n}(X_{s};s,t$ )$\allowbreak \neq \allowbreak p_{n}(X_{s};s)$ implying
that in general the transitional measure is not determined by the marginal
one completely. A relationship however exists. What is its nature? Can we
describe it?

\section{Proofs\label{dowody}}

\begin{proof}[Proof of Proposition \protect\ref{CTLI}]
i) Consider (\ref{podst}). By out assumption matrix $\mathbf{C}_{n}(t,s)$ is
non-singular hence both matrices $\mathbf{M}_{n}(s)$ and $\mathcal{A}%
_{n}(s,t)$ must be non-singular.

ii) By the tower property applied to $X_{u}^{n}$ we have :%
\begin{gather*}
\sum_{j=0}^{n}\gamma _{n,j}(s,u)X_{s}^{j}=E(X_{u}^{n}|\mathcal{F}_{\leq
s})=E(E(X_{u}^{n}|\mathcal{F}_{\leq t})|\mathcal{F}_{\leq s}) \\
=E(\sum_{j=0}^{n}\gamma _{n,j}(t,u)X_{t}^{j}|\mathcal{F}_{\leq
s})=\sum_{j=0}^{n}\gamma _{n,j}(t,u)\sum_{k=0}^{j}\gamma _{j,k}(s,t)X_{s}^{k}
\\
=\sum_{k=0}^{n}X_{s}^{k}\sum_{j=k}^{n}\gamma _{j,k}(s,t)\gamma _{n,j}(t,u).
\end{gather*}%
Comparing respective coefficients we get for $k\leq n:\gamma
_{n,k}(s,u)\allowbreak =\allowbreak \sum_{j=k}^{n}\gamma _{n,j}(t,u)\gamma
_{j,k}(s,t)$ this system of equations can written briefly in the matrix
form: $\mathcal{A}_{n}(t,u)\mathcal{A}_{n}(s,t).$

iii) We set say $s=\allowbreak 0$ in ii) and then by i) we get 
\begin{equation*}
\mathcal{A}_{n}(0,u)\mathcal{A}_{n}^{-1}(0,t)\allowbreak =\allowbreak 
\mathcal{A}_{n}(t,u),
\end{equation*}%
for all $t\leq u$. Thus it remains to define $V_{n}(u)\overset{df}{=}%
\mathcal{A}_{n}(0,u).$

iv) First of all let us notice that by MSC assumption diagonal entries of
matrix $\mathcal{A}_{n}(s,t)$ are positive for $s$ sufficiently close to $t$
and $s<t$. Secondly recall that diagonal elements of triangular matrix are
equal to its eigenvalues. Thirdly from the decomposition (\ref{fact}) it
follows that the product of diagonal entries of matrices $V_{n}$ and $%
V_{n}^{-1}$ are equal to diagonal elements of $\mathcal{A}_{n}(s,t)$. Hence
we can select diagonal entries of $V_{n}$ to be positive at least for $s<t$
close to $t$. On the other hand since by assumption elements of all matrices 
$V_{n},$ $\mathcal{A}_{n}(s,t)$ are continuous functions of $t$ they cannot
change sign since all matrices involved are non-singular.
\end{proof}

\begin{proof}[Proof of Corollary \protect\ref{ind_inc}]
First let us assume the process is with independent increments. Then from (%
\ref{fact}) it follows that 
\begin{equation}
\mathcal{A}_{n}(s,t)V_{n}(s)\allowbreak =\allowbreak V_{n}(t)  \label{eqn}
\end{equation}%
and we know from Proposition \ref{indep} that: 
\begin{equation*}
E(X_{t}^{n}|\mathcal{F}_{\leq s})=\sum_{j=0}^{n}\binom{n}{j}\gamma
_{n-j,0}(s,t)X_{s}^{j},
\end{equation*}%
since in this case $\gamma _{1,1}(s,t)\allowbreak =\allowbreak 1$ and $%
\gamma _{1,0}(s,t)\allowbreak =\allowbreak 0$. Hence the $i,j-$th entry of
the matrix $\mathcal{A}_{n}$ must be of the form $\binom{i}{j}\gamma
_{i-j,0}(s,t)$ where $\gamma _{i-j,0}$ is its $(i-j,0)-$th entry. Let us
denote $\gamma _{k,0}(s,t)$ by $\gamma _{k}(s,t)$ for brevity. Besides $i,i-$%
th entry of this matrix is equal to $1$. Assuming that diagonal entries of
matrix $V_{n}(s)$ are also $1$ we have to show that $v_{i,j}(t)\allowbreak
=\allowbreak \binom{i}{j}g_{i-j}(t)$ for $i>j$ for some functions $g_{k}(t)$%
. Secondly notice that $i,j-$th entry of the matrix equation (\ref{eqn})
takes the form 
\begin{equation}
\sum_{k=j}^{i}\binom{i}{k}\gamma _{i-k}(s,t)v_{k,j}(s)=v_{i,j}(t).
\label{pom}
\end{equation}

Besides since matrix $V_{N}$ is not defined uniquely we can take $%
V_{N}(0)\allowbreak =\allowbreak I_{N}$ -identity matrix. Then taking $%
j\allowbreak =\allowbreak i-1$ we see that $i\gamma _{1}(s,t)\allowbreak
=\allowbreak v_{i,i-1}(t)-v_{i,i-1}(s)$. Now we have also $\gamma
_{1}(s,t)\allowbreak =\allowbreak v_{1,0}(t)-v_{1,0}(s)$. Let us denote $%
v_{1,0}(t)\allowbreak =\allowbreak g_{1}(t)$. with $g_{1}(0)\allowbreak
=\allowbreak 0$. Comparing these two expressions we see that $\left(
v_{i,i-1}(t)-v_{i,i-1}(s)\right) /i\allowbreak =\allowbreak
g_{1}(t)-g_{1}(s).$ Hence we have proved that $v_{i,i-1}(t)\allowbreak
=\allowbreak \binom{i}{i-1}g_{1}(t)$. Further proof is by induction. Hence
assume that $v_{i,i-j}(t)\allowbreak =\allowbreak \binom{i}{j}g_{j}(t)$ for $%
j\leq m$. Let us take $j\allowbreak =\allowbreak i\allowbreak -m\allowbreak
\allowbreak -\allowbreak 1$. We have by the induction assumption:%
\begin{eqnarray*}
&&v_{i,i-m-1}(s)+\sum_{k=i-m-1}^{i-1}\binom{i}{i-k}\binom{k}{i-m-1}%
g_{k-i+m+1}(t)\gamma _{i-k}(s,t) \\
&=&v_{i,i-m-1}(s)+\binom{i}{m+1}\sum_{k=i-m-1}^{i-1}\binom{m+1}{i-k}%
g_{i-k}(t)\gamma _{k-i+m+1}(s,t)\allowbreak \\
&=&\allowbreak v_{i,i-m-1}(s)+\binom{i}{m+1}\sum_{s=0}^{m}\binom{m+1}{m-s}%
g_{m+1-s}(t)\gamma _{s}(s,t).
\end{eqnarray*}%
So $v_{i,i-m-1}(s)/\binom{i}{m+1}$ does not depend on $i$. So let us denote $%
v_{m+1,0}(s)\allowbreak =\allowbreak g_{m+1}(s)$. Consequently $%
v_{i,i-m-1}(s)\allowbreak =\allowbreak \binom{i}{m+1}g_{m+1}(s).$
\end{proof}

\begin{proof}[Proof of Theorem \protect\ref{lin_har}]
First let us notice that $\chi (dx,t|y,s,z,u)\allowbreak =\allowbreak \frac{%
\phi (x,t;y,s)\phi (z,u;x,t)}{\phi (z,u;y,s)}\mu (dx,t)$ is the conditional
distribution of $X_{t}|X_{s}=y,X_{u}=z$ for $s<t<u$. Let us find $%
E(X_{t}|X_{s}=y,X_{u}=z)$. We have using (\ref{trans}): 
\begin{equation*}
\chi (x,t|y,s,z,u)\allowbreak =\allowbreak \frac{1}{\allowbreak \phi
(z,u;y,s)}\mu (dx;t)\mu (dz;u)\sum_{j,k\geq 0}\frac{1}{\hat{p}_{j}(t)\hat{p}%
_{k}(u)}p_{j}(x;t)p_{j}(y,s)p_{k}(z;u)p_{k}(x;t).
\end{equation*}%
To perform our calculations swiftly let us notice that from (\ref{w_o}) it
follows that we have expansion 
\begin{equation}
p_{1}(x;t)p_{n}(x;t)=\hat{\alpha}_{n+1}(t)p_{n+1}(x;t)+\hat{\beta}%
_{n}(t)p_{n}(x;t)+\hat{\gamma}_{n-1}(t)p_{n-1}(x;t),  \label{3tr_mod}
\end{equation}%
where $\hat{\alpha}_{n+1}(t)\allowbreak =\allowbreak \alpha _{n+1}(t)/\alpha
_{1}(t);$ $\hat{\beta}_{n}(t)\allowbreak =\allowbreak (\beta _{n}(t)-\beta
_{0}(t))/\alpha _{1}(t);$ $\hat{\gamma}_{n-1}(t)\allowbreak =\allowbreak
\gamma _{n-1}(t)/\alpha _{1}(t)$. So we have 
\begin{gather*}
E(p_{1}(X_{t};t)|X_{s}=y,X_{u}=z)\allowbreak =\allowbreak \frac{1}{%
\sum_{n\geq 0}p_{n}(z;u)p_{n}(y,s)/\hat{p}_{n}(u)}\times \\
\sum_{j,k\geq 0}\frac{1}{\hat{p}_{j}(t)\hat{p}_{k}(u)}p_{j}(y,s)p_{k}(z;u)%
\int p_{1}(x;t)p_{j}(x;t)p_{k}(x;t)\mu (dx,t) \\
=\frac{1}{\sum_{n\geq 0}p_{n}(z;u)p_{n}(y,s)/\hat{p}_{n}(u)}\times \\
\sum_{j,k\geq 0}\frac{1}{\hat{p}_{j}(t)\hat{p}_{k}(u)}p_{j}(y,s)p_{k}(z;u)%
\int (\hat{\alpha}_{j+1}p_{j+1}(x;t)+\hat{\beta}_{j}(t)p_{j}(x;t)+\hat{\gamma%
}_{j-1}(t)p_{j-1}(x;t))p_{k}(x;t)\mu (dx,t).
\end{gather*}%
Let us calculate 
\begin{eqnarray*}
&&\sum_{j,k\geq 0}\frac{1}{\hat{p}_{j}(t)\hat{p}_{k}(u)}p_{j}(y,s)p_{k}(z;u)
\\
&&\times \int (\hat{\alpha}_{j+1}p_{j+1}(x;t)+\hat{\beta}_{j}(t)p_{j}(x;t)+%
\hat{\gamma}_{j-1}(t)p_{j-1}(x;t))p_{k}(x;t)\mu (dx,t)\allowbreak \\
&=&\allowbreak \sum_{k=0}^{\infty }\frac{\hat{\alpha}_{k}(t)\hat{p}_{k}(t)}{%
\hat{p}_{k-1}(t)\hat{p}_{k}(u)}p_{k-1}(y;s)p_{k}(z;u)\allowbreak
+\allowbreak \sum_{k=0}^{\infty }\hat{\beta}_{k}(t)p_{k}(y,s)p_{k}(z;u)/\hat{%
p}_{k}(u)\allowbreak \\
&&+\allowbreak \sum_{k=0}^{\infty }\frac{\hat{p}_{k}(t)\hat{\gamma}_{k}(t)}{%
\hat{p}_{k+1}(t)}p_{k+1}(y,s)p_{k}(z;u)/\hat{p}_{k}(u).
\end{eqnarray*}

Now using Proposition \ref{poczatek} iii) and definition of functions $\hat{%
\alpha}_{n}(t)$ , $\hat{\gamma}_{n}(t)$ and $\hat{\beta}_{n}(t)$ we have:$%
\frac{\hat{\alpha}_{k}(t)\hat{p}_{k}(t)}{\hat{p}_{k-1}(t)\hat{p}_{k}(u)}%
\allowbreak =\allowbreak \frac{\gamma _{k-1}(t)\hat{p}_{k-1}(t)}{\hat{p}%
_{k-1}(t)\hat{p}_{k}(u)\alpha _{1}(t)}\allowbreak =\allowbreak \frac{\hat{%
\gamma}_{k-1}(t)}{\hat{p}_{k}(u)}$ and similarly for other coefficients: 
\begin{gather*}
E(p_{1}(X_{t};t)|X_{s}=y,X_{u}=z)\allowbreak =\allowbreak \frac{1}{%
\sum_{n\geq 0}p_{n}(z;u)p_{n}(y,s)/\hat{p}_{n}(u)} \\
\times \sum_{k\geq 0}(\hat{\gamma}_{k-1}(t)p_{k-1}(y;s)+\hat{\beta}%
_{k}(t)p_{k}(y,s)+\hat{\alpha}_{k+1}(t)p_{k+1}(y,s))p_{k}(z;u)/\hat{p}%
_{k}(u).
\end{gather*}%
Now by assumption that $\mathbf{X}$ is a harness we must have the following
equality: 
\begin{gather*}
\sum_{k\geq 0}(\hat{\gamma}_{k-1}(t)p_{k-1}(y;s)+\hat{\beta}%
_{k}(t)p_{k}(y,s)+\hat{\alpha}_{k+1}(t)p_{k+1}(y,s))p_{k}(z;u)/\hat{p}_{k}(u)
\\
A(s,t,u)\sum_{n\geq 0}p_{n}(z;u)(\hat{\alpha}_{n+1}(s)p_{n+1}(y,s)+\hat{\beta%
}_{n}(s)p_{n}(y,s)+\hat{\gamma}_{n-1}(s)p_{n-1}(y,s))/\hat{p}_{n}(u) \\
+B(s,t,u)\sum_{n\geq 0}(\hat{\alpha}_{n+1}(u)p_{n+1}(z,u)+\hat{\beta}%
_{n}(u)p_{n}(z;u)+\hat{\gamma}_{n-1}(u)p_{n-1}(z,u))p_{n}(y,s)/\hat{p}%
_{n}(u).
\end{gather*}%
Let us multiply both sides of this equality by $p_{m}(z;u)$ and integrate
with respect to $\mu (dz,u)$. We will get%
\begin{gather*}
(\hat{\gamma}_{m-1}(t)p_{m-1}(y;s)+\hat{\beta}_{m}(t)p_{m}(y,s)+\hat{\alpha}%
_{m+1}(t)p_{m+1}(y,s)) \\
+B(s,t,u)(\hat{\alpha}_{m}(u)p_{m-1}(y;s)\hat{p}_{m}(u)/\hat{p}_{m-1}(u)+%
\hat{\beta}_{m}(u)p_{m}(y;s)+\hat{\gamma}_{m}(u)p_{m+1}(y;s)\hat{p}_{m}(u)/%
\hat{p}_{m+1}(u))
\end{gather*}%
Taking into account Proposition \ref{poczatek} iii) and uniqueness of
expansion in orthogonal polynomials we get the following equations%
\begin{eqnarray}
\hat{\beta}_{n}(t) &=&A(s,t,u)\hat{\beta}_{n}(s)+B(s,t,u)\hat{\beta}_{n}(u),
\label{beta} \\
\hat{\gamma}_{m-1}(t) &=&A(s,t,u)\hat{\gamma}_{m-1}(s)+B(s,t,u)\hat{\gamma}%
_{m-1}(u),  \label{gama} \\
\hat{\alpha}_{m+1}(t) &=&A(s,t,u)\hat{\alpha}_{m+1}(s)+B(s,t,u)\hat{\alpha}%
_{m+1}(u),  \label{alfa}
\end{eqnarray}%
where $\left\{ \alpha _{n}(t),\beta _{n}(t),\gamma _{n}(t)\right\} $ are the
coefficients of the modified 3-term recurrence (\ref{3tr_mod}) satisfied by
polynomials $\left\{ p_{n}\right\} $ and functions $A,$ $B$ are given by (%
\ref{_A}) and (\ref{_C}). To find functions $\beta _{n}(t),$ $\alpha _{n}(t)$
and $\gamma _{n}(t)$ satisfying equations (\ref{beta})-(\ref{alfa}) we will
use the following auxiliary result:

\begin{lemma}
\label{fun_eq}Let $g(t)$ be some nonzero, monotone continuous function and
suppose that continuous function $f(t)$ satisfies functional equation%
\begin{equation*}
f(t)=\frac{g(u)-g(t)}{g(u)-g(s)}f(s)+\frac{g(t)-g(s)}{g(u)-g(s)}f(u),
\end{equation*}%
for all $s\neq t\neq u\neq s$. Then $f(t)$ is a linear function of $g(t).$

\begin{proof}
Since $\frac{g(u)-g(t)}{g(u)-g(s)}\allowbreak +\allowbreak \frac{g(t)-g(s)}{%
g(u)-g(s)}\allowbreak =\allowbreak 1$ we have $\frac{f(t)-f(s)}{g(t)-g(s)}%
\allowbreak =\frac{f(u)-f(t)}{g(u)-g(t)}\allowbreak $. Hence $\frac{f(u)-f(t)%
}{g(u)-g(t)}$ does not depend on $u$. Hence $f(u)\allowbreak =\allowbreak
\xi (t)(g(u)\allowbreak -\allowbreak g(t))+f(t)$. Taking two different
values of $t$ say $t_{1}$ and $t_{2}$ we get: $0\allowbreak =\allowbreak
(\xi (t_{1})-\xi \left( t_{2}\right) )g(u)\allowbreak +\allowbreak
C(t_{1},t_{2})$ for all $u$. Since $g(u)$ is not constant we deduce that $%
\xi \left( t_{1}\right) \allowbreak =\allowbreak \xi (t_{2})$ and that $%
C(t_{1},t_{2})\allowbreak =\allowbreak 0$ which leads to conclusion that $%
f(t_{1})-\xi \left( t_{1}\right) g(t_{1})\allowbreak =\allowbreak
f(t_{2})-\xi \left( t_{2}\right) g(t_{2})$. Both these conclusions lead to
linearity of $f(t)$ with respect to $g(t).$
\end{proof}
\end{lemma}

Now recall that $A(s,t,u)\allowbreak =\allowbreak \frac{\hat{p}(u)-\hat{p}(t)%
}{\hat{p}(u)-\hat{p}(s)}$ and $B(s,t,u)\allowbreak =\allowbreak $ $\frac{%
\hat{p}(t)-\hat{p}(s)}{\hat{p}(u)-\hat{p}(s)}$. Hence we immediately have:%
\begin{equation*}
\hat{\beta}_{n}(t)\allowbreak =\allowbreak b_{n}+\hat{b}_{n}\hat{p}(t),~\hat{%
\gamma}_{n}(t)\allowbreak =\allowbreak c_{n}+\hat{c}_{n}\hat{p}(t),~\alpha
_{n}(t)=a_{n}+\hat{a}_{n}\hat{p}(t).
\end{equation*}%
Now it remains recall definitions of coefficients $\hat{\alpha},$ $\hat{\beta%
},$ $\hat{\gamma}.$
\end{proof}

\begin{proof}[Proof of Lemma \protect\ref{quadrat}]
Now let us consider expression for $E(p_{2}(X_{t};t)|\mathcal{F}_{s,u})$ .
On the way we will use the following notation:%
\begin{gather}
p_{2}(x;t)p_{n}(x;t)=r_{2,n+2}(t)p_{n+2}(x;t)+r_{1,n+1}(t)p_{n+1}(x;t)
\label{nw1} \\
+r_{0,n}(t)p_{n}(x;t)+r_{-1,n-1}(t)p_{n-1}(x;t)r_{-2,n-2}(t)p_{n-2}(x;t),
\label{nw2} \\
p_{1}(x;t)p_{n}(x;t)=v_{1,n+1}(t)p_{n+1}(x;t)+v_{0,n}(t)p_{n}(x;t)+v_{-1,n-1}(t)p_{n-1}(x;t).
\label{nw3}
\end{gather}%
During all calculations we use the following observations that we recall
here for the clarity of exposition. They follow properties of polynomials $%
p_{n}$ and Proposition \ref{poczatek} For $n\geq 0$ and $t,s\in I,s<t<u:$

\begin{gather*}
Ep_{n}(X_{t};t)p_{m}(X_{t};t)=\hat{p}_{n}(t)\delta _{n,m},~E(p_{n}(X_{t};t)|%
\mathcal{F}_{\leq s})=p_{n}(X_{s};s), \\
~E(p_{n}(X_{s};s)|\mathcal{F}_{\geq t})=\frac{\hat{p}_{n}(s)}{\hat{p}_{n}(t)}%
p_{n}(X_{t};t), \\
Ep_{1}^{2}(X_{t};t)p_{2}(X_{t};t)=\frac{\alpha _{2}(t)}{\alpha _{1}(t)}\hat{p%
}_{2}(t),~Ep_{1}^{3}(X_{t};t)=\frac{(\beta _{1}(t)-\beta _{0}(t)}{\alpha
_{1}(t)}\hat{p}(t) \\
Ep_{1}^{2}(X_{t};t)p_{1}^{2}(X_{s},s)=\frac{1}{\alpha _{1}(t)\alpha _{1}(s)}
\\
\times (\alpha _{2}(t)\alpha _{2}(s)\hat{p}_{2}(s)+(\beta _{1}(t)-\beta
_{0}(t))(\beta _{1}(s)-\beta _{0}(s))\hat{p}_{1}(s)+\gamma _{0}(t)\gamma
_{0}(s))
\end{gather*}

We will show that coefficients $A,$ $B,$ $C,$ $D,$ $E,$ $F$ satisfy the
following system of linear equations.:%
\begin{gather}
0=F+B\hat{p}(s),~\text{~}1=(A+C)+B\frac{\alpha _{2}(s)}{\alpha _{1}(s)},
\label{1} \\
~\frac{\gamma _{1}(t)}{\alpha _{2}(t)}\hat{p}(t)\allowbreak =\allowbreak A%
\frac{\gamma _{1}(s)}{\alpha _{2}(s)}\hat{p}(s)\allowbreak +\allowbreak C%
\frac{\gamma _{1}(u)}{\alpha _{2}(u)}\hat{p}(u)+B\frac{\gamma _{1}(u)}{%
\alpha _{1}(u)}\hat{p}(s),  \label{2} \\
0=\frac{\beta _{1}(s)-\beta _{0}(s)}{\alpha _{1}(s)}B+(D+E),~~0\allowbreak
=\allowbreak \frac{\beta _{1}(u)-\beta _{0}(u)}{\alpha _{1}(u)}B\hat{p}(s)+D%
\hat{p}(s)+E\hat{p}(u),  \label{3} \\
\frac{\gamma _{1}(t)}{\alpha _{1}(t)}\allowbreak =\allowbreak A\frac{\gamma
_{1}(s)}{\alpha _{1}(s)}+C\frac{\gamma _{1}(u)}{\alpha _{1}(u)}+B(\frac{%
\alpha _{2}(u)\gamma _{1}(s)}{\alpha _{1}(u)\alpha _{1}(s)}+\frac{(\beta
_{1}(u)-\beta _{0}(u))(\beta _{1}(s)-\beta _{0}(s))}{\alpha _{1}(u)\alpha
_{1}(s)}  \label{4} \\
+\hat{p}(u))+D\frac{\beta _{1}(s)-\beta _{0}(s)}{\alpha _{1}(s)}+E\frac{%
\beta _{1}(u)-\beta _{0}(u)}{\alpha _{1}(u)}+F.  \notag
\end{gather}

In our lengthy calculations we will process formula (\ref{QH}) which we copy
here below for the convenience of the reader.%
\begin{equation*}
E(p_{2}(X_{t};t)|\mathcal{F}%
_{s,u})=Ap_{2}(X_{s};s)+Bp_{1}(X_{s};s)p_{1}(X_{u};u)+Cp_{2}(X_{u};u)+Dp_{1}(X_{s};s)+Ep_{1}(X_{u};u)+F.
\end{equation*}%
To get first assertion of equation (\ref{1}) we integrate (\ref{QH}). To get
the second one we multiply (\ref{QH}) by $p_{2}(X_{s};s)$ and integrate. To
get (\ref{2}) we multiply (\ref{QH}) by $p_{2}(X_{u};u)$ and integrate. To
get first assertion of (\ref{3}) we multiply (\ref{QH}) by $p_{1}(X_{s};s)$
and integrate. To get the second one we multiply (\ref{QH}) by $%
p_{1}(X_{u};u)$ and integrate. To get (\ref{4}) we multiply (\ref{QH}) by $%
p_{1}(X_{s};s)p_{1}(X_{u};u)$ and integrate. Next we use identities: $\alpha
_{2}(t)\hat{p}_{2}(t)\allowbreak =\allowbreak \gamma _{1}(t)\hat{p}(t)$ and $%
\alpha _{1}(t)\hat{p}(t)\allowbreak =\allowbreak \gamma _{0}(t)$ and then
cancel out $\hat{p}(s).$

Now following Theorem \ref{lin_har} and Proposition \ref{poczatek} to prove
second assertion we have to take into account the following facts: 
\begin{eqnarray*}
\beta _{1}(t)-\beta _{0}(t) &=&b\alpha (t)+\hat{b}\gamma (t)=\alpha (t)(b+%
\hat{b}\hat{p}(t)), \\
\alpha _{2}(t) &=&a\alpha (t)+\hat{a}\gamma (t)=\alpha (t)(a+\hat{a}\hat{p}%
(t)) \\
\gamma _{1}(t) &=&c\alpha (t)+\hat{c}\gamma (t)=\alpha (t)(c+\hat{c}\hat{p}%
(t)), \\
\hat{p}(t) &=&\hat{p}_{1}(t)=\frac{\gamma (t)}{\alpha (t)}, \\
\hat{p}_{2}(t) &=&\frac{\gamma _{1}(t)\hat{p}(t)}{\alpha _{2}(t)}=\frac{%
\gamma (t)(c\alpha \left( t\right) +\hat{c}\gamma (t))}{\alpha (t)(a\alpha
\left( t\right) +\hat{a}\gamma (t)}=\hat{p}(t)\frac{c+\hat{c}\hat{p}(t)}{a+%
\hat{a}\hat{p}(t)},
\end{eqnarray*}%
where we denoted for simplicity $\alpha (t)\allowbreak =\allowbreak \alpha
_{1}(t)$ and $\gamma (t)\allowbreak =\allowbreak \gamma _{0}(t)$ and $a$, $%
\hat{a},$ $b,$ $\hat{b},$ $c,$ $\hat{c}$ are numerical parameters with
defined above meaning. On the way we divide when it is necessary both sides
by $\hat{p}\left( s\right) $ since $\hat{p}(t)$ is nonzero as an expectation
of a square bracket of the martingale $p_{1}(X_{t};t)$. So equations (\ref{1}%
)-(\ref{3}) now become: 
\begin{gather}
F=-B\hat{p}(s),~1=A+C+B(a+\hat{a}\hat{p}\left( s\right) ),  \label{__1} \\
\hat{p}(t)\frac{(c+\hat{c}\hat{p}(t))}{(a+\hat{a}\hat{p}(t)}=A\hat{p}(s)%
\frac{(c+\hat{c}\hat{p}(s))}{(a+\hat{a}\hat{p}(s))}+C\hat{p}(u)\frac{(c+\hat{%
c}\hat{p}(u))}{(a+\hat{a}\hat{p}(u)}+B\hat{p}(s)(c+\hat{c}\hat{p}(u)),
\label{__2} \\
0=D+E+B(b+\hat{b}\hat{p}(s)),~~0=D\hat{p}(s)+E\hat{p}(u)+B\hat{p}(s)(b+\hat{b%
}\hat{p}(u))  \label{__3} \\
(c+\hat{c}\hat{p}(t))=A(c+\hat{c}\hat{p}(s))+C(c+\hat{c}\hat{p}(u))+D(b+\hat{%
b}\hat{p}(s))+E(b+\hat{b}\hat{p}(u))  \label{__4} \\
-B\hat{p}(s)+B((a+\hat{a}\hat{p}(s))(a+\hat{a}\hat{p}(u))\frac{(c+\hat{c}%
\hat{p}(s))}{(a+\hat{a}\hat{p}(s)}+(b+\hat{b}\hat{p}(s))(b+\hat{b}\hat{p}%
(u))+\hat{p}(u)).  \label{__5}
\end{gather}%
Now notice that using identities (\ref{__2}) we get; 
\begin{gather*}
D(b+\hat{b}\hat{p}(s))+E(b+\hat{b}\hat{p}(u))=b(D+E)+\hat{b}(D\hat{p}(s)+E%
\hat{p}(u)) \\
=-bB(b+\hat{b}\hat{p}(s))-\hat{p}(s)B\hat{b}(b+\hat{b}\hat{p}(u)) \\
=-B(b^{2}+2b\hat{b}\hat{p}(s)+\hat{b}^{2}\hat{p}(s)\hat{p}(u)).
\end{gather*}%
Further we have: 
\begin{eqnarray*}
&&B(b+\hat{b}\hat{p}(s))(b+\hat{b}\hat{p}(u)-B(2b\hat{b}\hat{p}(s)+b^{2}+%
\hat{b}^{2}\hat{p}(s)\hat{p}(u))\allowbreak  \\
&=&\allowbreak B(b^{2}+b\hat{b}\hat{p}(s)+b\hat{b}\hat{p}(u)+\hat{b}^{2}\hat{%
p}(s)\hat{p}(u)-2b\hat{b}\hat{p}(s)-b^{2}-\hat{b}^{2}\hat{p}(s)\hat{p}%
(u))\allowbreak  \\
&=&\allowbreak Bb\hat{b}(\hat{p}(u)-\hat{p}(s)).
\end{eqnarray*}

Hence our system of $5$ equations can be split into two sets: The set
consisting of three equations satisfied by unknowns $A,$ $B,$ $C$ : 
\begin{eqnarray*}
1 &=&A+C+B(a+\hat{a}\hat{p}(s)),~ \\
\hat{p}(t)\frac{c+\hat{c}\hat{p}(t)}{a+\hat{a}\hat{p}(t)} &=&A\hat{p}(s)%
\frac{(c+\hat{c}\hat{p}(s))}{(a+\hat{a}\hat{p}(s))}+C\hat{p}(u)\frac{(c+\hat{%
c}\hat{p}(u))}{(a+\hat{a}\hat{p}(u))}+B\hat{p}(s)(c+\hat{c}\hat{p}(u)), \\
c+\hat{c}\hat{p}(t) &=&A(c+\hat{c}\hat{p}(s))+C(c+\hat{c}\hat{p}(u))+B((\hat{%
p}(u)-\hat{p}(s))(1+b\hat{b})+(a+\hat{a}\hat{p}(u))(c+\hat{c}\hat{p}(s))
\end{eqnarray*}%
The other set of two equations that is satisfied by $D$ and $E$. 
\begin{equation*}
0=D+E+B(b+\hat{b}\hat{p}(s)),~~0=D\hat{p}(s)+E\hat{p}(u)+B\hat{p}(s)(b+\hat{b%
}\hat{p}(u)).
\end{equation*}%
The first set of equations yields:%
\begin{eqnarray*}
B &=&\frac{a\hat{c}\lambda (\hat{p}(t)-\hat{p}(s))(\hat{p}(u)-\hat{p}(t))}{%
(a+\hat{a}\hat{p}(t))(\hat{p}(u)-\hat{p}(s))(\hat{p}(s)\hat{p}(u)\hat{a}\hat{%
c}\kappa +\hat{p}(u)a\hat{c}\kappa +\hat{p}(s)a\hat{c}(\kappa -\lambda
)+ac\kappa )}, \\
A &=&\frac{(a+\hat{a}\hat{p}(s))(\hat{p}(u)-\hat{p}(t))(\hat{p}(u)\hat{p}(t)%
\hat{a}\hat{c}\kappa +\hat{p}(u)a\hat{c}\kappa +\hat{p}(t)a\hat{c}(\kappa
-\lambda )+ac\kappa )}{(a+\hat{a}\hat{p}(t))(\hat{p}(u)-\hat{p}(s))(\hat{p}%
(s)\hat{p}(u)\hat{a}\hat{c}\kappa +\hat{p}(u)a\hat{c}\kappa +\hat{p}(s)a\hat{%
c}(\kappa -\lambda )+ac\kappa )}, \\
C &=&\frac{(a+\hat{a}\hat{p}(u))(\hat{p}(t)-\hat{p}(s))(\hat{p}(t)\hat{p}(s)%
\hat{a}\hat{c}\kappa +\hat{p}(s)a\hat{c}(\kappa -\lambda )+\hat{p}(t)a\hat{c}%
\kappa +ac\kappa )}{(a+\hat{a}\hat{p}(t))(\hat{p}(u)-\hat{p}(s))(\hat{p}(s)%
\hat{p}(u)\hat{a}\hat{c}\kappa +\hat{p}(u)a\hat{c}\kappa +\hat{p}(s)a\hat{c}%
(\kappa -\lambda )+ac\kappa )}.
\end{eqnarray*}

where we denoted $\kappa \allowbreak =\allowbreak (1+b\hat{b}+\hat{a}c),$ $%
\lambda \allowbreak =\allowbreak (a\hat{c}-\hat{a}c)$

As far as the other system of equations is concerned we have: $0=D+E+B(b+%
\hat{b}\hat{p}(s)),~~0=D\hat{p}(s)+E\hat{p}(u)+B\hat{p}(s)(b+\hat{b}\hat{p}%
(u))$ which gives%
\begin{equation*}
D=-bB(s,t,u),~E=-\hat{b}B(s,t,u)\hat{p}(s)
\end{equation*}
\end{proof}

\begin{proof}[Proof of Theorem \protect\ref{main}]
The proof is based on several auxiliary results of which the most important
is the following lemma:
\end{proof}

\begin{lemma}
\label{aux}Let $\mathbf{X=(}X_{t})_{t\in I}$ be a Markov process with
orthogonal polynomial martingales $\left\{ p_{n}(x;t)\right\} _{n\geq -1}$
and one dimensional marginal distribution $\mu (dx,t)$ and $\eta (dx,t;y,s)$
as transitional distribution satisfying Condition \ref{sq_int}. Then $%
\mathbf{X}$ being a harness is a quadratic harness iff the following system
of equations is satisfied:%
\begin{eqnarray*}
r_{-2,m-2}(t) &=&Ar_{-2,m-2}(s)+Cr_{-2,m-2}(u)+Bv_{-1,m-2}(s)v_{-1,m-1}(u),
\\
r_{-1,m-1}(t) &=&Ar_{-1,m-1}(s)+Cr_{-1,m-1}(u)+B(v_{0,m-1}(s)v_{-1,m-1}(u) \\
&&+v_{-1,m-1}(s)v_{0,m}(u))+Dv_{-1,m-1}(s)+Ev_{-1,m-1}(u), \\
r_{0,m}(t) &=&Ar_{0,m}(s)+Cr_{0,m}(u)+B(v_{-1,m-1}(u)v_{1,m}(s) \\
&&+v_{0,m}(u)v_{0,m}(s)+v_{1,m+1}(u)v_{-1,m}(s))+Dv_{0,m}(s)+Ev_{0,m}(u), \\
r_{1,m+1}(t) &=&Ar_{1,m+1}(s)+Cr_{1,m+1}(u)+B(v_{0,m}(u)v_{1,m+1}(s) \\
&&+v_{1,m+1}(u)v_{0,m+1}(s))+Dv_{1.m+1}(s)+Ev_{1,m+1}(u), \\
r_{2,m+2}(t) &=&Ar_{2,m+2}(s)+Cr_{2,m+2}(u)+Bv_{1,m+2}(s)v_{1,m+1}(u),
\end{eqnarray*}%
where $A,$ $B,$ $C,$ $D,$ $E,$ are defined by (\ref{bb})-(\ref{de}) and
coefficients $r_{i,j},$ $i\allowbreak =\allowbreak -2,\ldots ,2$ , $%
j\allowbreak =\allowbreak m-2,\ldots ,m+2$ are defined by (\ref{nw1})-(\ref%
{nw3}).
\end{lemma}

\begin{proof}[Proof of Lemma \protect\ref{aux}]
We have :%
\begin{gather*}
E(p_{2}(X_{t};t)|\mathcal{F}_{s,u})=\frac{1}{\frac{1}{\sum_{n\geq
0}p_{n}(z;u)p_{n}(y,s)/\hat{p}_{n}(u)}}\times \\
\sum_{j,k\geq 0}\frac{1}{\hat{p}_{j}(t)\hat{p}_{k}(u)}p_{j}(y,s)p_{k}(z;u)%
\int p_{2}(x;t)p_{j}(x;t)p_{k}(x;t)\mu (dx,t) \\
=\frac{1}{\sum_{n\geq 0}p_{n}(z;u)p_{n}(y,s)/\hat{p}_{n}(u)}\times \\
\sum_{j,k\geq 0}\frac{1}{\hat{p}_{j}(t)\hat{p}_{k}(u)}p_{j}(y,s)p_{k}(z;u)%
\times \\
\int
(r_{2,j+2}p_{j+2}+r_{1,j+1}(t)p_{j+1}+r_{0,j}(t)p_{j}+r_{-1,j-1}p_{j-1}+r_{-2,j-2}p_{j-2})p_{k}(x;t)\mu (dx,t).
\end{gather*}%
Now notice that:%
\begin{eqnarray*}
&&\sum_{j,k\geq 0}\frac{1}{\hat{p}_{j}(t)\hat{p}_{k}(u)}p_{j}(y,s)p_{k}(z;u)
\\
&&\times \int
(r_{2,j+2}p_{j+2}+r_{1,j+1}(t)p_{j+1}+r_{0,n}(t)p_{j}+r_{-1,j-1}p_{j-1}+r_{-2,j-2}p_{j-2})p_{k}(x;t)\mu (dx,t)\allowbreak
\\
&=&\allowbreak \sum_{k=0}^{\infty }\frac{r_{2,k}(t)\hat{p}_{k}(t)}{\hat{p}%
_{k-2}(t)\hat{p}_{k}(u)}p_{k-2}(y;s)p_{k}(z;u)\allowbreak
+\sum_{k=0}^{\infty }\frac{r_{1,k}(t)\hat{p}_{k}(t)}{\hat{p}_{k-1}(t)\hat{p}%
_{k}(u)}p_{k-1}(y;s)p_{k}(z;u)\allowbreak \allowbreak \\
&&+\allowbreak \sum_{k=0}^{\infty }r_{0,k}(t)p_{k}(y,s)p_{k}(z;u)/\hat{p}%
_{k}(u)\allowbreak \\
&&+\allowbreak \sum_{k=0}^{\infty }\frac{r_{-1,k}(t)\hat{p}_{k}(t)}{\hat{p}%
_{k+1}(t)\hat{p}_{k}(u)}p_{k+1}(y;s)p_{k}(z;u)\allowbreak \\
&&+\allowbreak \sum_{k=0}^{\infty }\frac{\hat{p}_{k}(t)r_{-2,k}(t)}{\hat{p}%
_{k+2}(t)}p_{k+2}(y,s)p_{k}(z;u)/\hat{p}_{k}(u).
\end{eqnarray*}%
Using (\ref{uproszcz}) we get 
\begin{gather*}
E(p_{2}(X_{t};t)|\mathcal{F}_{s,u})=\frac{1}{\frac{1}{\sum_{n\geq
0}p_{n}(z;u)p_{n}(y,s)/\hat{p}_{n}(u)}}\times \\
(\sum_{k=0}^{\infty
}(r_{-2,k-2}(t)p_{k-2}(y;s)+r_{-1,k}(t)p_{k-1}(y;s)+r_{0,k}(t)p_{k}(y,s) \\
+r_{1,k+1}(t)p_{k+1}(y,s)+r_{2,k+2}(t)p_{k+2}(y,s))p_{k}(z;u)/\hat{p}_{k}(u).
\end{gather*}%
Now let assume that $E(p_{2}(X_{t};t)|\mathcal{F}_{s,u})\allowbreak
=\allowbreak
Ap_{2}(X_{s};s)+Bp_{1}(X_{s};s)p_{1}(X_{u};u)+Cp_{2}(X_{u};u)+Dp_{1}(X_{s};s)+Ep_{1}(X_{u};u)+F 
$. This implies the the following equality: 
\begin{gather*}
(\sum_{k=0}^{\infty
}(r_{-2,k-2}(t)p_{k-2}(y;s)+r_{-1,k}(t)p_{k-1}(y;s)+r_{0,k}(t)p_{k}(y,s) \\
+r_{1,k+1}(t)p_{k+1}(y,s)+r_{2,k+2}(t)p_{k+2}(y,s))p_{k}(z;u)/\hat{p}_{k}(u)=
\end{gather*}%
\begin{gather*}
=(Ap_{2}(y;s)+Bp_{1}(y;s)p_{1}(z;u)+Cp_{2}(z;u)+Dp_{1}(y;s)+Ep_{1}(z;u)+F) \\
\times \sum_{n\geq 0}p_{n}(z;u)p_{n}(y,s)/\hat{p}_{n}(u)=-p(s)B\sum_{n\geq
0}p_{n}(z;u)p_{n}(y,s)/\hat{p}_{n}(u) \\
+A\sum_{n\geq 0}\frac{p_{n}(z,u)}{\hat{p}_{n}(u)}%
(r_{2,n+2}(s)p_{n+2}(y,s)+r_{1,n+1}(s)p_{n+1}(y,s)+r_{0,n}(s)p_{n}(y,s) \\
+r_{-1,n-1}(s)p_{n-1}(y,s)+r_{-2,n-1}(s)p_{n-2}(y,s)) \\
+C\sum_{n\geq 0}\frac{p_{n}(y,s)}{\hat{p}_{n}(u)}%
(r_{2,n+2}(u)p_{n+2}(z,u)+r_{1,n+1}(u)p_{n+1}(z,u) \\
+r_{0,n}(u)p_{n}(z,u)+r_{-1,n-1}(u)p_{n-1}(z,u)+r_{-2,n-1}(u)p_{n-2}(z,u)) \\
+B\sum_{n\geq 0}\frac{1}{\hat{p}_{n}(u)}%
(v_{1,n+1}(s)p_{n+1}(y,s)+v_{0,n}(s)p_{n}(y,s)
\end{gather*}%
\begin{gather*}
+v_{-1,n-1}(s)p_{n-1}(y;s))(v_{1,n+1}(u)p_{n+1}(z,u)+v_{0,n}(u)p_{n}(z,u)+v_{-1,n-1}(u)p_{n-1}(z;u))
\\
+D\sum_{n\geq 0}\frac{p_{n}(z,u)}{\hat{p}_{n}(u)}%
(v_{1,n+1}(s)p_{n+1}(y,s)+v_{0,n}(s)p_{n}(y,s)+v_{-1,n-1}(s)p_{n-1}(y;s)) \\
+E\sum_{n\geq 0}\frac{p_{n}(y,s)}{\hat{p}_{n}(u)}%
(v_{1,n+1}(u)p_{n+1}(z,u)+v_{0,n}(u)p_{n}(z,u)+v_{-1,n-1}(u)p_{n-1}(z;u))
\end{gather*}%
Now let us multiply both sides of this equality by $p_{m}(z;u)$ and
integrate with respect to distribution $\mu (dz,u)$. We get:%
\begin{gather*}
(r_{-2,m-2}(t)p_{m-2}(y;s)+r_{-1,m}(t)p_{m-1}(y;s)+r_{0,m}(t)p_{m}(y,s) \\
+r_{1,m+1}(t)p_{m+1}(y,s)+r_{2,m+2}(t)p_{m+2}(y,s)) \\
=-\hat{p}(s)Bp_{m}(y,s)+A(r_{2,m+2}(s)p_{m+2}(y,s)+r_{1,m+1}(s)p_{m+1}(y,s)
\\
+r_{0,m}(s)p_{m}(y,s)+r_{-1,m-1}(s)p_{m-1}(y,s)+r_{-2,m-1}(s)p_{m-2}(y,s))
\end{gather*}%
\begin{gather*}
+C(\frac{p_{m-2}(y,s)}{\hat{p}_{m-2}(u)}r_{2,m}(u)\hat{p}_{m}\left( u\right)
+\frac{p_{m-1}(y,s)}{\hat{p}_{m-1}(u)}r_{1,m}(u)\hat{p}%
_{m}(u)+p_{m}(y;s)r_{0,m} \\
+\frac{p_{m+1}(y,s)}{\hat{p}_{m+1}(u)}r_{-1,m-1}(u)\hat{p}_{m-1}(u)+\frac{%
p_{m+2}(y,s)}{\hat{p}_{m+2}(u)}r_{-2,m-2}(u)\hat{p}_{m-2}(u)) \\
+D(v_{1,m+1}(s)p_{m+1}(y,s)+v_{0,m}(s)p_{m}(y,s)+v_{-1,m-1}(s)p_{m-1}(y;s))
\\
+E(\frac{p_{m-1}(y,s)}{\hat{p}_{m-1}(u)}v_{1,m}(u)\hat{p}%
_{m}(u)+v_{0,m}(u)p_{m}(y,s)+\frac{p_{m+1}(y,s)}{\hat{p}_{m+1}(u)}%
v_{-1,m-1}(u)\hat{p}_{m-1}(u)) \\
+B(\frac{v_{1,m}(u)\hat{p}_{m}(u)}{\hat{p}_{m-1}(u)}%
(v_{1,m}(s)p_{m}(y;s)+v_{0,m-1}(s)p_{m-1}(y;s)+v_{-1,m-2}(s)p_{m-2}(y;s)) \\
+v_{0,m}(u)((v_{1,m+1}(s)p_{m+1}(y,s)+v_{0,m}(s)p_{m}(y,s)+v_{-1,m-1}(s)p_{m-1}(y;s))
\\
+\frac{v_{-1,m}(u)\hat{p}_{m}(u)}{\hat{p}_{m+1}(u)}%
(v_{1,m+2}(s)p_{m+2}(y,s)+v_{0,m+1}(s)p_{m+1}(y;s)+v_{-1,m}(s)p_{m}(y;s)).
\end{gather*}%
Taking into account (\ref{uproszcz1}-\ref{uproszcz3}) and uniqueness of
expansion in orthogonal polynomials we get the following equations: $%
r_{-2,n-2}(t)\hat{p}_{n-2}(t)=r_{2,n}(t)\hat{p}_{n}(t),~r_{-1,n-1}(t)\hat{p}%
_{n-1}(t)=r_{1,n-1}(t)\hat{p}_{n-1}(t),~v_{-1,n-1}(t)\hat{p}%
_{n-1}(t)=v_{1,n}(t)\hat{p}_{n}(t).$

Now to solve system of equations given in Lemma \ref{aux} we need several
following technical results.
\end{proof}

\begin{proposition}
\label{pomoc}Coefficients $r$ and $v$ are related to one another by the
following formulae:%
\begin{eqnarray*}
r_{2,n+2}(t)\allowbreak &=&\frac{\alpha _{1}(t)}{\alpha _{2}(t)}%
v_{1,n+1}(t)v_{1,n+2}(t),~r_{1,n+1}(t)=v_{1,n+1}(\frac{\alpha _{1}}{\alpha
_{2}}(v_{0,n+1}+v_{0,n})-\frac{\beta _{1}-\beta _{0}}{\alpha _{2}}), \\
r_{0,n}(t) &=&(\frac{\alpha _{1}}{\alpha _{2}}(v_{1,n+1}v_{-1,n}\allowbreak
+\allowbreak v_{0,n}v_{0,n}+v_{-1,n-1}v_{1,n})\allowbreak -\allowbreak \frac{%
\beta _{1}-\beta _{0}}{\alpha _{2}}v_{0,n}-\frac{\gamma _{0}}{\alpha _{2}}),
\\
r_{-1,n-1}(t) &=&\frac{\alpha _{1}}{\alpha _{2}}%
v_{-1,n-1}(v_{0,n}+v_{0,n-1})-\frac{\beta _{1}-\beta _{0}}{\alpha _{2}}%
v_{-1,n-1},~r_{-2,n-2}(t)\allowbreak =\allowbreak \frac{\alpha _{1}}{\alpha
_{2}}v_{-1,n-1}v_{-1,n-2}.
\end{eqnarray*}%
Besides we have: 
\begin{eqnarray}
r_{-2,n-2}(t)\hat{p}_{n-2}(t) &=&r_{2,n}(t)\hat{p}_{n}(t),  \label{uproszcz1}
\\
~r_{-1,n-1}(t)\hat{p}_{n-1}(t) &=&r_{1,n-1}(t)\hat{p}_{n-1}(t),
\label{uproszcz2} \\
~v_{-1,n-1}(t)\hat{p}_{n-1}(t) &=&v_{1,n}(t)\hat{p}_{n}(t).
\label{uproszcz3}
\end{eqnarray}
Coefficients $v$ are related to coefficients of the 3 term recurrence by
formulae:%
\begin{equation*}
v_{-1,n-1}(t)=\frac{\gamma _{n-1}(t)}{\alpha _{1}(t)},~v_{0,n}(t)=\frac{%
\beta _{n}(t)}{\alpha _{1}(t)},~v_{1,n+1}(t)=\frac{\alpha _{n+1}(t)}{\alpha
_{1}(t)}.
\end{equation*}
\end{proposition}

\begin{proof}
We have: $%
p_{1}^{2}p_{n}=p_{1}(v_{1,n+1}p_{n+1}+v_{0,n}p_{n}+v_{-1,n-1}p_{n-1})=v_{1,n+1}v_{1,n+2}p_{n+2}+v_{1,n+1}v_{0,n+1}p_{n+1}+v_{1,n+1}v_{-1,n}p_{n}\allowbreak +\allowbreak v_{0,n}v_{1,n+1}p_{n+1}\allowbreak +\allowbreak v_{0,n}v_{0,n}p_{n}+v_{0,n}v_{-1,n-1}p_{n-1}\allowbreak +\allowbreak v_{-1,n-1}v_{1,n}p_{n}\allowbreak +\allowbreak v_{-1,n-1}v_{0,n-1}p_{n-1}\allowbreak +\allowbreak v_{-1,n-1}v_{-1,n-2}p_{n-2}\allowbreak =\allowbreak v_{1,n+1}v_{1,n+1}p_{n+2}+(v_{1,n+1}v_{0,n+1}+v_{0,n}v_{1,n+1})p_{n+1}\allowbreak +\allowbreak (v_{1,n+1}v_{-1,n}\allowbreak +\allowbreak v_{0,n}v_{0,n}+v_{-1,n-1}v_{1,n})p_{n}\allowbreak +\allowbreak (v_{0,n}v_{-1,n-1}+v_{-1,n-1}v_{0,n-1})p_{n-1}\allowbreak +\allowbreak v_{-1,n-1}v_{-1,n-2}p_{n-2}\allowbreak =v_{1,n+1}v_{1,n+1}p_{n+2}\allowbreak +\allowbreak v_{1,n+1}(v_{0,n+1}+v_{0,n})\allowbreak +\allowbreak (v_{1,n+1}v_{-1,n}\allowbreak +\allowbreak v_{0,n}v_{0,n}+v_{-1,n-1}v_{1,n})p_{n}\allowbreak +\allowbreak v_{-1,n-1}(v_{0,n}+v_{0,n-1})p_{n-1}\allowbreak +\allowbreak v_{-1,n-1}v_{-1,n-2}p_{n-2} 
$. We dropped arguments to simplify calculations. Remembering that: $p_{2}=%
\frac{\alpha _{1}}{\sigma _{2}}p_{1}^{2}-\frac{\left( \beta _{1}-\beta
_{0}\right) }{\alpha _{2}}p_{1}-\frac{\gamma _{0}}{\alpha _{2}}$ we have:

$p_{2}p_{n}\allowbreak =\allowbreak \frac{\alpha _{1}}{\sigma _{2}}%
(v_{1,n+1}v_{1,n+1}p_{n+2}\allowbreak +\allowbreak
v_{1,n+1}(v_{0,n+1}+v_{0,n})\allowbreak +\allowbreak
(v_{1,n+1}v_{-1,n}\allowbreak +\allowbreak
v_{0,n}v_{0,n}+v_{-1,n-1}v_{1,n})p_{n}\allowbreak +\allowbreak
v_{-1,n-1}(v_{0,n}+v_{0,n-1})p_{n-1}\allowbreak +\allowbreak
v_{-1,n-1}v_{-1,n-2}p_{n-2})\allowbreak -\allowbreak \frac{\beta _{1}-\beta
_{0}}{\alpha _{2}}(v_{1,n+1}p_{n+1}+v_{0,n}p_{n}+v_{-1,n-1}p_{n-1})-\frac{%
\gamma _{0}}{\alpha _{2}}p_{n}\allowbreak =\allowbreak p_{n+2}\frac{\alpha
_{1}}{\sigma _{2}}v_{1,n+1}v_{1,n+1}+p_{n+1}(\frac{\alpha _{1}}{\alpha _{2}}%
v_{1,n+1}(v_{0,n+1}+v_{0,n})-v_{1,n+1}\frac{\beta _{1}-\beta _{0}}{\alpha
_{2}})\allowbreak +\allowbreak p_{n}(\frac{\alpha _{1}}{\alpha _{2}}%
(v_{1,n+1}v_{-1,n}\allowbreak +\allowbreak
v_{0,n}v_{0,n}+v_{-1,n-1}v_{1,n})\allowbreak -\allowbreak \frac{\beta
_{1}-\beta _{0}}{\alpha _{2}}v_{0,n}-\frac{\gamma _{0}}{\alpha _{2}}%
)\allowbreak +\allowbreak p_{n-1}(\frac{\alpha _{1}}{\sigma _{2}}%
v_{-1,n-1}(v_{0,n}+v_{0,n-1})-\frac{\beta _{1}-\beta _{0}}{\alpha _{2}}%
v_{-1,n-1})\allowbreak +\allowbreak \frac{\alpha _{1}}{\alpha _{2}}%
v_{-1,n-1}v_{-1,n-2}p_{n-2}$. Comparing coefficients by $p_{n+2},$ $p_{n+1},$
$p_{n},$ $p_{n-1},$ and $p_{n-2}$ we get our assertion. Now let us assume
that we deal with harness that is following assertion of Theorem \ref%
{harness} we assume that:%
\begin{eqnarray*}
\frac{\alpha _{1}(t)}{\alpha _{2}(t)} &=&\frac{1}{a+\hat{a}p(t)},~\frac{%
\gamma _{0}(t)}{\alpha _{2}(t)}=\frac{\frac{\gamma _{0}(t)}{\alpha _{1}(t)}}{%
\frac{\alpha _{2}(t)}{\alpha _{1}(t)}}=\frac{p(t)}{a+\hat{a}p(t)},\frac{%
\beta _{1}(t)-\beta _{0}(t)}{\alpha _{2}(t)}=\frac{\frac{\beta _{1}(t)-\beta
_{0}(t)}{\alpha _{1}(t)}}{\frac{\alpha _{2}(t)}{\alpha _{1}(t)}}=\frac{b+%
\hat{b}p(t)}{a+\hat{a}p(t)}, \\
v_{-1,n-1}(t) &=&c_{n-1}+\hat{c}_{n-1}p(t),~v_{0,n}(t)=b_{n}+\hat{b}%
_{n}p(t),~v_{1,n+1}(t)=a_{n+1}+\hat{a}_{n+1}p(t),
\end{eqnarray*}%
with $c_{0}=c,$ $\hat{c}_{0}=\hat{c},$ $b_{1}=b,$ $\hat{b}_{1}=\hat{b},$ $%
a_{1}=a,$ $\hat{a}_{1}=\hat{a}.$
\end{proof}

\begin{proof}
Now to prove Theorem \ref{main} we combine equations from the assertion of
Lemma \ref{aux}, simplifications from Lemma \ref{pomoc} and formulae for
coefficients $A,$ $B,$ $C,$ $D,$ $E,$ $F$ given by Lemma \ref{quadrat}. We
collect all expressions on one side of these equations and factor out (with
the help of Mathematica). It turns out that this factorization is of the
following form : Certain expression involving sequences $a_{n},$ $\hat{a}%
_{n},$ $b_{n},$ $\hat{b}_{n},$ $c_{n}$ and $\hat{c}_{n}$ times some
expressions depending on $t$. Hence to satisfy our equations expressions in
this factorization that do not depend on $t$ must be set to zero.
\end{proof}


\begin{thebibliography}{99}
\bibitem{Biane98} Biane, Philippe. Processes with free increments.\emph{\
Math. Z.} \textbf{227} (1998), no. 1, 143--174. MR1605393 (99e:46085)

\bibitem{brwe05} Bryc, W. , Weso\l owski, J. (2005), Conditional Moments of $%
q$-Meixner Processes, \emph{Probab. Theory Rel. Fields} \textbf{131}, 415-441

\bibitem{Bo} Bo\.{z}ejko, Marek; K\"{u}mmerer, Burkhard; Speicher, Roland. $%
q-$-Gaussian processes: non-commutative and classical aspects. \emph{Comm.
Math. Phys.} \textbf{185} (1997), no. 1, 129--154. MR1463036 (98h:81053)

\bibitem{Bryc2001M} Bryc, Wlodzimierz. Stationary Markov chains with linear
regressions. \emph{Stochastic Process. Appl.} \textbf{93} (2001), no. 2,
339--348. MR1828779 (2002d:60058)

\bibitem{Bryc2001S} Bryc, W\l odzimierz. Stationary random fields with
linear regressions. \emph{Ann. Probab.} \textbf{29} (2001), no. 1, 504--519.
MR1825162 (2002d:60014)

\bibitem{bms} Bryc, W\l odzimierz; Matysiak, Wojciech; Szab\l owski, Pawe\l\ %
J. Probabilistic aspects of Al-Salam-Chihara polynomials. \emph{Proc. Amer.
Math. Soc.} \textbf{133} (2005), no. 4, 1127--1134 (electronic). MR2117214
(2005m:33033)

\bibitem{BryWe} Bryc, W, Weso\l owski, J. (2007), Bi - Poissson process,
Infinite Dimensional Analysis, \emph{Quantum Probability and Related Topics} 
\textbf{10} (2) , 277-291

\bibitem{BryBo} Bo\.{z}ejko, M. Bryc, W. (2006), On a Class of Free Levy
Laws Related to a Regression Problem, \emph{Journal of Functional Analysis} 
\textbf{236} , 59-77.

\bibitem{BryMaWe} Bryc, W. , Matysiak, W. , Weso\l owski, J. , The Bi -
Poisson process: a quadratic harness. \emph{Annals of Probability} \textbf{36%
} (2) (2008), s. 623-646

\bibitem{BryMaWe07} Wlodzimierz Bryc, Wojciech Matysiak, Jacek Wesolowski.
Quadratic Harnesses, $q-$commutations, and orthogonal martingale
polynomials. \emph{Trans. Amer. Math. Soc.} \textbf{359} (2007), no. 11,
5449--5483

\bibitem{BryWe10} W\l odek Bryc, Jacek Weso\l owski. Askey--Wilson
polynomials, quadratic harnesses and martingales. IMS-AOP-AOP503. \emph{%
Annals of Probability} 2010, Vol. \textbf{38}, No. 3, 1221-1262.

\bibitem{BRWE12} Bryc, W\l odzimierz, Weso\l owski Jacek, Stitching pairs of
Levy processes into harnesses, \emph{Stochastic Processes and their
Applications}, 122(2012), str. 2854 - 2869

\bibitem{BryMaWe11} Wlodzimierz Bryc, Wojciech Matysiak, Jacek Weso\l owski,
Free Quadratic Harness. \emph{Stochastic Processes and their Applications }%
\textbf{121} (2011) 657--671.

\bibitem{Cuh12} Cuchiero, Christa; Keller-Ressel, Martin; Teichmann, Josef.
Polynomial processes and their applications to mathematical finance. \emph{%
Finance Stoch.} \textbf{16} (2012), no. 4, 711--740. MR2972239 (Reviewed)

\bibitem{Eagleson64} Eagleson, G. K. Polynomial expansions of bivariate
distributions. \emph{Ann. Math. Statist.} \textbf{35} 1964 1208--1215.
MR0168055 (29 \#5320)

\bibitem{Ham67} Hammersley, J. M. Harnesses. \emph{1967 Proc. Fifth Berkeley
Sympos. Mathematical Statistics and Probability} (Berkeley, Calif.,
1965/66), Vol. III: Physical Sciences pp. 89--117 Univ. California Press,
Berkeley, Calif. MR0224144 (36 \#7190)

\bibitem{Jaco88} Jacod, Jean; Protter, Philip. Time reversal on L\'{e}vy
processes. \emph{Ann. Probab.} \textbf{16} (1988), no. 2, 620--641.
MR0929066 (89j:60102)

\bibitem{Koek} Koekoek, Roelof; Lesky, Peter A.; Swarttouw, Ren\'{e} F.
Hypergeometric orthogonal polynomials and their \$q\$-analogues. With a
foreword by Tom H. Koornwinder. \emph{Springer Monographs in Mathematics.
Springer-Verlag}, Berlin, 2010. xx+578 pp. ISBN: 978-3-642-05013-8 MR2656096
(2011e:33029)

\bibitem{Lancaster58} H. O. Lancaster, The structure of bivariate
distributions, \emph{Ann. Math. Statistics}, vol. \textbf{29}, no. 3, pp.
719-736, September 1958.

\bibitem{Lancaster63(1)} H. O. Lancaster, Correlations and canonical forms
of bivariate distributions, \emph{Ann. Math. Statistics}, vol. \textbf{34},
no. 2, pp. 532-538, June 1963.

\bibitem{Lancaster63(2)} H. O. Lancaster, Correlation and complete
dependence of random variables, \emph{Ann. Math. Statistics}, vol. 34, no.
4, pp. 1315-1321, December 1963.

\bibitem{Lancaster75} Lancaster, H. O. Joint probability distributions in
the Meixner classes. \emph{J. Roy. Statist. Soc. Ser. B }\textbf{37} (1975),
no. 3, 434--443. MR0394971 (52 \#15770)

\bibitem{Sim98} Simon, Barry. The classical moment problem as a self-adjoint
finite difference operator. \emph{Adv. Math.} 137 (1998), no. 1, 82--203.
MR1627806 (2001e:47020)

\bibitem{Scho98} Schoutens, Wim; Teugels, Jozef L. L\'{e}vy processes,
polynomials and martingales. Special issue in honor of Marcel F. Neuts. 
\emph{Comm. Statist. Stochastic Models} \textbf{14} (1998), no. 1-2,
335--349. MR1617536 (99k:60118)

\bibitem{Szabl-exp} Szab\l owski, Pawe\l\ J. Expansions of one density via
polynomials orthogonal with respect to the other. \emph{J. Math. Anal. Appl.}
\textbf{383} (2011), no. 1, 35--54. MR2812716, http://arxiv.org/abs/1011.1492

\bibitem{Szab-OU-W} Szab\l owski, Pawe\l\ J. $q-$Wiener and $(\alpha ,q)-$
Ornstein--Uhlenbeck processes. A generalization of known processes, in print
in \emph{Theory of Probability and Its Applications, }%
2011,http://arxiv.org/abs/math/0507303

\bibitem{SzablAW} Szab\l owski, Pawe\l\ J. On the structure and
probabilistic interpretation of Askey-Wilson densities and polynomials with
complex parameters. \emph{J. Funct. Anal.} \textbf{261} (2011), no. 3,
635--659. MR2799574, http://arxiv.org/abs/1011.1541

\bibitem{SzablKer} Szab\l owski, Pawe\l\ J., On summable, positive
Poisson-Mehler kernels built of Al-Salam--Chihara and related polynomials, 
\emph{Infin. Dimens. Anal. Quantum Probab. Relat. Top,}Vol. \textbf{15, }No.
3(2012), 1250014 (18pp), http://arxiv.org/abs/1011.1848

\bibitem{SzabChol} Szab\l owski, Pawe\l\ J., A few remarks on orthogonal
polynomials, submitted, http://arxiv.org/abs/1207.1172

\bibitem{SzabHar} Pawe\l\ J. Szab\l owski, A few remarks on quadratic
harnesses, http://arxiv.org/abs/1207.1172

\bibitem{Yor05} Mansuy, Roger; Yor, Marc. Harnesses, L\'{e}vy bridges and
Monsieur Jourdain. \emph{Stochastic Process. Appl.} 115 (2005), no. 2,
329--338. MR2111197 (2005m:60104)
\end{thebibliography}
\end{document}